\documentclass[a4paper,leqno,12pt]{amsart}
\usepackage[utf8]{inputenc}

\usepackage[top=25mm, bottom=25mm, left=20mm, right=20mm]{geometry}

\usepackage{amsmath, amstext, amssymb, enumitem, mathtools}

\usepackage{booktabs, array}

\usepackage[hidelinks]{hyperref}
\hypersetup{
    colorlinks = true,
    urlcolor   = blue,
    linkcolor  = blue,
    citecolor  = red
}

\newcommand{\NN}{\mathbb{N}}
\newcommand{\RR}{\mathbb{R}}

\newcommand{\me}{\mathrm{e}}
\newcommand{\dint}{\,\mathrm{d}}

\DeclarePairedDelimiter{\norm}{\lVert}{\rVert}
\DeclarePairedDelimiter{\abs}{\lvert}{\rvert}
\DeclarePairedDelimiter{\inner}{\langle}{\rangle}

\newcommand{\TT}{\mathbb{T}}
\newcommand{\ZZ}{\mathbb{Z}}
\newcommand{\Hd}{\dot H^1}
\newcommand{\Hdd}{\dot H^2}
\newcommand{\LL}{\mathcal{L}}
\newcommand{\Knet}{K^{\mathrm{net}}}
\newcommand{\Fheat}{\mathcal{F}_{\mathrm{heat}}}
\newcommand{\Fmix}{\mathcal{F}_{\mathrm{mix}}}
\newcommand{\Nsinc}{N_{\mathrm{sinc}}}
\newcommand{\Acl}{\mathcal{A}_\theta}
\newcommand{\AB}{\mathcal{A}_{\mathcal B}}
\newcommand{\Mcal}{\mathcal{M}}
\newcommand{\Ecal}{\mathcal{E}}
\newcommand{\HR}{\textup{(H$_R$)}}
\newcommand{\one}{\mathbf{1}}
\newcommand{\BE}{\mathrm{BE}}

\numberwithin{equation}{section}

\theoremstyle{plain}
\newtheorem{theorem}[equation]{Theorem}
\newtheorem{proposition}[equation]{Proposition}
\newtheorem{lemma}[equation]{Lemma}
\newtheorem{corollary}[equation]{Corollary}
\newtheorem*{thmA}{Theorem A}
\newtheorem*{thmB}{Theorem B}
\newtheorem*{thmC}{Theorem C}

\theoremstyle{definition}
\newtheorem{definition}[equation]{Definition}
\newtheorem{assumption}[equation]{Assumption}
\newtheorem{remark}[equation]{Remark}

\begin{document}

\title[Heat-semigroup emulation of the Darcy operator]{Emulating the Darcy Solution Operator\\ by Heat-Semigroup Functional Calculus\\ and Pointwise Products}

\author{Dawid Hanrahan}
\address{
Dawid Hanrahan (\textnormal{dawid.hanrahan@pwr.edu.pl})\\
Wroc{\l}aw University of Science and Technology,
50-370 Wroc{\l}aw, Poland
}

\begin{abstract}
We approximate the solution operator $(a,f)\mapsto u$ of the Darcy equation $-\operatorname{div}(a\nabla u)=f$ by operator networks whose only primitives are functions of one fixed reference operator---heat-semigroup evaluations $\me^{-tL_0}$ of the Laplacian or, more generally, of the generator of a diffusion---and pointwise products.
Functions of $L_0$ commute with translations and cannot by themselves see a variable coefficient; the obstacle is the coupling $\nabla a\cdot\nabla u$.
We show that a single heat layer already produces it: the \emph{product defect} $[P_t(bu)-P_tb\,P_tu]/(2t)$ converges to $\nabla b\cdot\nabla u$ as $t\to0$ with a quantitative rate.
Inserted into the Laplacian-preconditioned Richardson iteration, this yields an explicit scheme of depth $O(\log(1/\varepsilon)/\log(1/q))$, with every parameter fixed by the coefficient bounds and the accuracy $\varepsilon$, whose $\Hd$-error is at most $\varepsilon$ uniformly over the coefficient class and over $L^2$ data; each block is controlled in operator norm on the continuum, without Galerkin truncation, and on the torus the scheme applied to band-limited inputs is evaluated exactly on a grid of $O(1/(\varepsilon a_{\min}))$ points per direction, up to a logarithm, the class constants and, for the $C^2$-type class, a factor $\sqrt d$.
Because the argument is carr\'e du champ calculus, the theorem holds on strongly local Dirichlet spaces with a spectral gap and a Bakry--\'Emery bound $\BE(-K,\infty)$, with parameters that depend on the space only through the spectral-gap normalisation and $K$, not on its dimension; on the torus it yields a $C^2$-type coefficient class with dimension-free constants, and a refinement to Wiener classes of fractional regularity $1+\theta$ with a rate $t^{\theta/2}$ that the product-defect term attains.
The mechanism has a price, which we quantify: the factor $1/(2t)$ that extracts a derivative from a semigroup makes the scheme, with the proved constants, ill-conditioned at small $\varepsilon$.


\end{abstract}

\maketitle

\section{Introduction}
\label{sec:intro}

\subsection{The question}
\label{subsec:question}
Let $\TT^d=(\RR/2\pi\ZZ)^d$ with the normalised Lebesgue measure, let $a$ be a measurable coefficient with $0<a_{\min}\le a\le a_{\max}$, and let $f\in L^2(\TT^d)$ have mean zero.
The Darcy equation $-\operatorname{div}(a\nabla u)=f$ has a unique mean-zero weak solution $u=G(a,f)$ in $\Hd$, the space of mean-zero functions with $\norm{u}_1:=\norm{\nabla u}_{L^2}<\infty$.
The map $G\colon(a,f)\mapsto u$, the Darcy solution operator, is the standard test case in the approximation theory of operators between function spaces \cite{KLM21,Kov23,MS23,MS24}.
We ask a question about \emph{primitives}: which operations suffice to represent $G$ to accuracy $\varepsilon$ in $\Hd$, uniformly over a class of coefficients and over the unit ball of $L^2$ data, with every parameter explicit?

The primitives we allow are functions $\psi(L_0)$ of a single \emph{fixed} reference operator, here $L_0=-\Delta$---finite sums of heat-semigroup evaluations $P_t=\me^{-tL_0}$ and, in a variant, one resolvent $J_s=(I+sL_0)^{-1}$---together with pointwise products of two functions (``channels'') and pointwise affine maps.
An \emph{operator network} is a finite composition of such maps on finitely many channels (Definition~\ref{def:class}); $\Fheat$ denotes the family of finite exponential sums $\sum_ja_j\me^{-t_j\lambda}$ and $\Fmix$ the same family with one resolvent allowed.
The question is of interest for two reasons.
First, the approximation theory of operator networks has largely been built on primitives that see coordinates (Fourier modes, grids, coordinate gradients); on a manifold or a general diffusion space these are not available, but diffusion with respect to a fixed reference operator is, and closely related linear layers appear as manifold filters \cite{WRR24}, with learned diffusion times in DiffusionNet \cite{SACO22}, and as Laplace--Beltrami spectral layers in \cite{CLM24,WH25,FOW26}.
Lanthaler, Li and Stuart \cite{LLS25} show that nonlocal averaging with a pointwise nonlinearity is universal in a qualitative sense; the question here is quantitative and concerns one operator.
Semigroup-driven unrolled Picard iterations with quantitative bounds exist for semilinear parabolic equations \cite{FTO25}, where the semigroup acts on a pointwise nonlinearity of the state; here it has to produce a differential coupling between coefficient and state.
Second, the answer turns out to be an \emph{algorithm}: an unrolled preconditioned iteration whose every step is realised inside the primitive set, with a per-step error certified in operator norm, so that its one intrinsic limitation---a loss of conditioning at small $\varepsilon$---is of a kind numerical analysts know how to read.

\subsection{The mechanism: a product of heat layers is a derivative}
\label{subsec:mechanism}
The obstacle is immediate.
Every function of $L_0$ commutes with translations, whereas $G(a,\cdot)$ for a variable coefficient does not; so the coefficient can enter only through pointwise operations.
Moreover, when $-\operatorname{div}(a\nabla\cdot)$ is compared with $L_0$, the term that carries the coefficient is the first-order coupling $\nabla a\cdot\nabla u$---a \emph{differential} expression that no function of $L_0$ and no product of functions can produce from $a$ and $u$ separately.

The observation at the centre of this paper is that one heat layer already contains this coupling.
The heat semigroup does not respect products: for the Laplacian,
\begin{equation}
    \label{eq:intro-interpolation}
    P_t(bu)-P_tb\,P_tu=2\int_0^tP_s\bigl(\nabla P_{t-s}b\cdot\nabla P_{t-s}u\bigr)\dint s,
\end{equation}
the semigroup interpolation identity \cite{BGL14}, whose right-hand side is, to leading order in $t$, $2t\,\nabla b\cdot\nabla u$.
Hence the \emph{product defect} of one heat layer,
\begin{equation}
    \label{eq:intro-gamma}
    \Gamma_t(b,u):=\frac1{2t}\bigl[P_t(bu)-P_tb\,P_tu\bigr]\ \longrightarrow\ \Gamma(b,u):=\nabla b\cdot\nabla u\qquad(t\to0),
\end{equation}
synthesises the coupling from one heat time and pointwise products.
The defect is not new as an object---the small-time limit of such kernel expressions is the semigroup form of the carr\'e du champ \cite{BGL14}, and kernel-covariance estimators of $\Gamma$ serve in diffusion geometry (\cite{JL26}, building on \cite{Jon24}) to compute gradients from data; what is new is its quantitative control in the norm the iteration needs, and its use inside an explicit architecture.
On the torus the statement is exact in Fourier coefficients: $\Gamma_t(b,u)$ and $\Gamma(b,u)$ have the same multiplier up to a factor $F_t(k,\mu)\in(0,1]$ with $1-F_t\le\min\{1,2t(\abs{k}^2+\abs{\mu}^2)\}$ (Lemma~\ref{lem:product-defect}), so the two agree on the band $t(\abs{k}^2+\abs{\mu}^2)\ll1$ and never differ by more than the multiplier of $\Gamma$ itself.
The quantitative content of the paper is a bound on $\Gamma_t-\Gamma$ in the dual norm of $\Hd$: of order $\sqrt t$ when $\nabla b$ and $\Delta b$ are bounded (Theorem~\ref{thm:gamma-consistency}), and of order $t^{\theta/2}$ on Wiener classes of regularity $1+\theta$, where this rate is attained by the product-defect term (Section~\ref{sec:torus}).
The notation $\Gamma$ is not an accident: for the Laplacian, $\nabla b\cdot\nabla u$ is the \emph{carr\'e du champ} of the semigroup, and \eqref{eq:intro-interpolation} holds, with its own $\Gamma$, on the strongly local Dirichlet spaces of Section~\ref{sec:mechanism} (Lemma~\ref{lem:interpolation}).
Everything that follows is $\Gamma$-calculus; that is why the results hold beyond the torus.

\subsection{Laplacian preconditioning and the scheme}
\label{subsec:scheme}
Set $\bar a=\tfrac12(a_{\min}+a_{\max})$, $q=(a_{\max}-a_{\min})/(a_{\max}+a_{\min})<1$ and $b:=a/\bar a-1$, so that $\norm{b}_\infty\le q$.
With $K_b:=L_0^{-1}(-\operatorname{div}(b\nabla\cdot))$ on $\Hd$ the Darcy problem reads $\bar a(I+K_b)u=L_0^{-1}f$ with $\norm{K_b}_{\LL(\Hd)}\le q$, so that $G(a,f)=\bar a^{-1}\sum_{j\ge0}(-K_b)^jL_0^{-1}f$.
Truncating this Neumann series is the Laplacian-preconditioned Richardson iteration with the classical midpoint normalisation $\bar a$; it goes back to Concus and Golub \cite{CG73}, and its spectral analysis is in \cite{NHT09,GMNS19}.
Unrolled with $O(\log(1/\varepsilon))$ steps, iterations of this kind are the backbone of the Darcy emulations in \cite[Algorithm~23, Theorem~26]{KLM21} (a Picard iteration preconditioned by $(-\Delta)^{-1}$ in a Fourier--Galerkin space) and in \cite[Section~4]{MS24} (a Richardson iteration preconditioned by a fixed reference coefficient in a Galerkin snapshot space).
We analyse one step inside the primitive set.

The step has a structure that dictates the block.
On $\Hd$ the preconditioned operator is \emph{multiplication plus smoothing},
\begin{equation}
    \label{eq:intro-rep}
    K_b=\Pi M_b+L_0^{-1}\Pi\bigl(\Gamma(b,\cdot)-M_{L_0b}\bigr),
\end{equation}
where $M_b$ is multiplication by $b$ and $\Pi$ removes the mean (Theorem~\ref{thm:representation}; the identity is three lines via the Leibniz rule, and it is the classical structure behind Laplacian preconditioning \cite{FMP90,AK07,NHT09}).
The multiplication part is local, and the architecture performs it exactly.
Only the remainder has to be approximated---and it is smoothing, so the approximate inverse of $L_0$ is needed only as a map $L^2\to\Hd$ that is accurate relative to $L_0^{-1}$ on a spectral band.
Replacing each ingredient of \eqref{eq:intro-rep} by a primitive gives the block
\begin{equation}
    \label{eq:block-intro}
    \Knet u=(I-P_T)(bu)+R_m\bigl(\Gamma_t(b,u)-u\,b_\tau\bigr),\qquad b_\tau=(I-P_\tau)b/\tau,
\end{equation}
in which $I-P_T$ replaces $\Pi$, $\Gamma_t$ replaces $\Gamma$, $b_\tau$ replaces $L_0b$, and $R_m=r(L_0)$---a sinc-quadrature exponential sum of heat layers, or one resolvent---replaces $L_0^{-1}$.
The scheme is the unrolled iteration
\begin{equation}
    \label{eq:intro-scheme}
    v_0=R_mf,\qquad v_{i+1}=R_mf-\Knet v_i\quad(0\le i<N),\qquad N_\varepsilon(a,f)=\bar a^{-1}v_N,
\end{equation}
with all blocks sharing their parameters.
The error of one block splits into five terms, one per replacement, each controlled by one parameter; the per-step error propagates through the contraction, and $N=O(\log(1/\varepsilon)/\log(1/q))$ steps give accuracy $\varepsilon$ (Section~\ref{sec:mechanism}).
The map $N_\varepsilon$ is linear in $f$, a polynomial of degree $N$ in $a$ for fixed $f$, and uses no activation function.

\subsection{Results}
\label{subsec:results}
The first result is concrete and complete on the torus.
Throughout, $\kappa=a_{\max}/a_{\min}$, so that $1-q=2/(\kappa+1)$.

\begin{thmA}[$\TT^d$, $C^2$-type coefficients; Corollary~\ref{cor:torus-C2}]
Let $0<a_{\min}\le a_{\max}$, $B\ge1$ and $\varepsilon\in(0,1)$, and let $\AB(B)$ be the class of measurable $a$ with $a_{\min}\le a\le a_{\max}$ a.e.\ and $b=a/\bar a-1\in W^{2,\infty}(\TT^d)$, $\norm{\nabla b}_\infty+\norm{\Delta b}_\infty\le B$.
There is an explicit operator network $N_\varepsilon$ with $\Fheat$-layers (or with $\Fmix$-layers) and input channels $(a,f)$, of depth
\[
    N=\max\Bigl\{0,\Bigl\lceil\frac{\log\frac4{\varepsilon\bar a(1-q)}}{\log(1/q)}\Bigr\rceil\Bigr\}\qquad(N=0\text{ if }q=0),
\]
whose parameters are functions of $B$, $a_{\min}$, $a_{\max}$ and $\varepsilon$ alone---in particular independent of $d$---such that
\[
    \sup_{a\in\AB(B)}\ \sup_{\norm{f}_{L^2}\le1,\ \int f=0}\ \norm{G(a,f)-N_\varepsilon(a,f)}_{\Hd}\le\varepsilon .
\]
$N_\varepsilon$ is the iteration \eqref{eq:intro-scheme} with the block \eqref{eq:block-intro}; its parameters are those of recipe~C$^2$ in Table~\ref{tab:recipes}, with $K=0$.
\end{thmA}

Theorem~A is the case $K=0$ of the main theorem, which is stated on a general diffusion space under the three hypotheses the proof uses: a \emph{spectral gap}, so that $L_0^{-1}$ is bounded from $L^2$ to $\Hd$ and the band hypothesis on the approximate inverse makes sense; \emph{strong locality}, that is, the Leibniz rule for $\Gamma$, which gives \eqref{eq:intro-rep} and is exactly what fails on graphs; and a \emph{curvature lower bound} in the sense of Bakry--\'Emery, $\BE(-K,\infty)$---the abstract form of ``Ricci $\ge-K$''---used only through the commutation of gradient and heat flow up to a factor $\me^{Kt}$.

\begin{thmB}[strongly local Dirichlet spaces; Theorem~\ref{thm:dirichlet-main}]
Let $(X,\mu,\Ecal)$ be a strongly local Dirichlet space with carr\'e du champ $\Gamma$, generator $L_0$ with spectral gap $\lambda_1\ge1$, and $\BE(-K,\infty)$ for some $K\ge0$ (Assumption~\ref{ass:M}), and let $\AB(B)$ be the class of $a$ with $a_{\min}\le a\le a_{\max}$ and $\norm{\sqrt{\Gamma(b)}}_\infty+\norm{L_0b}_\infty\le B$.
Then the conclusion of Theorem~A holds for the Darcy problem $\int a\Gamma(u,v)\dint\mu=\int fv\dint\mu$, with the same scheme and the same depth $N$, and with parameters that are functions of $B$, $a_{\min}$, $a_{\max}$, $\varepsilon$ and $K$ alone---not of $X$, $\mu$, a dimension, or the spectrum of $L_0$ beyond $\lambda_1\ge1$.
\end{thmB}

The setting contains compact Riemannian manifolds with Ricci curvature bounded below (Section~\ref{subsec:dirichlet-facts}).
Only $\norm{b}_\infty$, $\norm{\sqrt{\Gamma(b)}}_\infty$ and $\norm{L_0b}_\infty$ enter---no bound on $L_0^2b$---because the product-defect estimate is formulated in the dual norm of $\Hd$; the curvature enters the parameters only through a factor $\me^{-2Kt_0}$ in the heat times, with $t_0\lesssim\varepsilon^2$.

The third result is a refinement on the torus, where Fourier analysis reaches coefficients with fewer than two derivatives and gives a sharpness statement for the consistency rate of the product defect.

\begin{thmC}[$\TT^d$, fractional regularity; Theorem~\ref{thm:main}]
Let $\theta\in(0,1]$ and let $\Acl$ be the class of $a$ with $a_{\min}\le a\le a_{\max}$ and $\norm{b}_{A^{1+\theta}}\le B$ in the Wiener norm $\norm{b}_{A^\sigma}=\sum_{\mu\in\ZZ^d}(1+\abs{\mu})^\sigma\abs{\hat b_\mu}$.
Then the conclusion of Theorem~A holds over $\Acl$, with the same depth $N$ and with parameters that are functions of $\theta$, $B$, $a_{\min}$, $a_{\max}$ and $\varepsilon$ alone (recipe~W in Table~\ref{tab:recipes}), the heat time of the product defect being of order $(\varepsilon\bar a(1-q)^2/B)^{2/\theta}$.
The rate $t^{\theta/2}$ of the product-defect term in the one-step bound is attained over the $A^{1+\theta}$ ball (Remark~\ref{rem:sharpness}).
\end{thmC}

For $\theta=1$, Theorem~A covers a strictly larger class than Theorem~C---the $C^2$-type class is not contained in any ball of $A^2(\TT^d)$---with the same depth and much smaller parameters (Remark~\ref{rem:recipe-comparison}); Theorem~C is the statement for $\theta<1$ and for the sharpness of the consistency rate.

\subsection{What is new}
\label{subsec:new}
\begin{enumerate}[label=(\alph*)]
    \item \emph{The mechanism, made quantitative.} The coupling $\nabla a\cdot\nabla u$---and hence a variable-coefficient elliptic solution operator---can be synthesised from functional calculus of a fixed operator and pointwise products, through the product defect of a single heat layer. The defect itself is the classical finite-time form of the carr\'e du champ; what is new is its control in the dual norm of $\Hd$ under $\norm{\nabla b}_\infty+\norm{L_0b}_\infty$, with a rate that the product-defect term attains in the Wiener scale, and its use to synthesise the Darcy coupling inside an explicit architecture.
    \item \emph{Geometry independence.} The emulation theorem holds on strongly local Dirichlet spaces under a spectral gap and $\BE(-K,\infty)$, with explicit constants and with curvature dependence confined to a factor $\me^{-2Kt_0}$: the reference diffusion carries the geometry, the pointwise products carry the coefficient.
    \item \emph{Continuum operator-norm analysis.} Each inexact step is controlled in the operator norm of $\Hd$ over $L^2$ data, without Galerkin truncation, and the approximate inverse is needed only on a spectral band, the rest being paid for by a tail term. On the torus the scheme applied to band-limited inputs is evaluated exactly by a pseudo-spectral realisation whose grid is set by the data and the coefficient, not by the internal scales of the construction (Corollary~\ref{cor:bandlimited}).
    \item \emph{The statements.} $L^2$ data, explicit parameters, uniformity over the stated classes, constants independent of $d$; on the torus a $C^2$-type class that is not contained in any Wiener ball, and Wiener classes of fractional regularity $1+\theta$.
\end{enumerate}
The classical ingredients---the unrolled Richardson iteration, the structure ``preconditioned operator $=$ multiplication-type operator $+$ compact remainder'', exponential-sum approximations of $1/\lambda$ \cite{BH05,BH26}, the interpolation identity \cite{BGL14}, and the $\Gamma$-calculus of $\BE(K,\infty)$ \cite{AGS15}---are attributed where they enter (Section~\ref{subsec:richardson}, Remark~\ref{rem:R-classical}); polylogarithmic parameter counts for Darcy-type operators are known for analytic data \cite{MS23}, \cite[Sections~5--6]{MS24}.

\subsection{What the results do not say, and what they cost}
\label{subsec:scope}
The results are approximation (expressivity) theorems for continuum operators with prescribed parameters.
They say nothing about trained networks, learnability, or computational efficiency after discretisation, and they do not claim superiority to classical solvers or dimension-free arithmetic complexity; ``independent of $d$'' refers to the constants for the stated coefficient classes.

\emph{Scope.}
The coefficient is scalar and uniformly positive, and the coefficient classes are regular: Theorems~A and~B need bounded $\nabla b$ and $L_0b$, Theorem~C needs $1+\theta$ derivatives in the Wiener sense; discontinuous coefficients, the common benchmark in operator learning, are not covered.
The data may be rough, $f\in L^2$.
Graph Laplacians, which have no Leibniz rule, are excluded (Section~\ref{subsec:dirichlet-facts}).

\emph{Contrast.}
All parameters are polynomial in $\kappa=a_{\max}/a_{\min}$.
Since $\log(1/q)\sim2/\kappa$ and $\bar a(1-q)=a_{\min}$, the depth is $N\approx(\kappa/2)\log(4/(\varepsilon a_{\min}))$ for large $\kappa$, the step tolerance is $\eta_0=\varepsilon a_{\min}/(4(\kappa+1))$ in the generic case, and the multiplier $1/(2t)$ of the recipe of Theorem~A is $200B^2(\kappa+1)^2/(\varepsilon a_{\min})^2$ (Section~\ref{sec:realisations}); at high contrast the construction is an approximation-theoretic statement, not a competitive solver.

\emph{Conditioning} is the one cost intrinsic to the product-defect block, and it is quantified in Section~\ref{sec:realisations}: accuracy $\eta_0$ of the product defect forces $\sqrt t\lesssim\eta_0/B$ (Remark~\ref{rem:sharpness}), and the block multiplies by $1/(2t)$, so about $\log_{10}(1/(2t))$ digits are lost to cancellation; with the proved constants the scheme is realisable in double precision only down to $\varepsilon\approx10^{-3}$ for the recipe of Theorem~A and $\varepsilon\approx10^{-2}$ for that of Theorem~C on the test problem of Section~\ref{sec:numerics}.
The cancellation is the price of extracting a derivative from a semigroup, and the resolvent variant does not remove it.

\emph{Discretisation} is not a cost of the mechanism: the internal scales of the construction---the band $\sqrt\Lambda$ of the approximate inverse and the heat scale $t^{-1/2}$ of the product defect---enter the proofs only through a tail term and the heat-time terms, not as frequencies a grid must resolve, and on the torus the scheme applied to band-limited inputs is evaluated exactly on a grid set by the data and the coefficient, of size $O(1/(\varepsilon a_{\min}))$ per direction up to a logarithm, the class constants and, for the $C^2$-type class, a factor $\sqrt d$ (Corollary~\ref{cor:bandlimited}); what remains open is the variant with per-step truncation (Section~\ref{sec:conclusion}).

\subsection{Organisation}
Section~\ref{sec:mechanism} explains the construction and proves its setting-independent parts.
Section~\ref{sec:dirichlet} proves Theorem~B and derives Theorem~A; Section~\ref{sec:torus} proves Theorem~C and the sharpness of the consistency rate.
Section~\ref{sec:realisations} collects the counts, the discretisation corollary and the floating-point analysis; Section~\ref{sec:numerics} is a numerical sanity check of both recipes and of the discretisation corollary; Section~\ref{sec:conclusion} concludes.
Appendix~\ref{app:sinc} proves the sinc-quadrature lemma; Appendix~\ref{app:aux} collects auxiliary lemmas on Dirichlet spaces, the non-inclusion example and the Jackson-mean lemma used in Corollary~\ref{cor:bandlimited}.

\section{From heat layers to the Darcy operator}
\label{sec:mechanism}

This section explains the construction and proves the parts of the argument that do not depend on the setting.
After it, each theorem of the paper has the same shape: a parameter recipe that makes the five error terms of one block sum to a prescribed $\eta_0$, inserted into one error-budget proposition.

\subsection{Setting and notation}
\label{subsec:setting}
Two settings are treated in parallel: the torus, and strongly local Dirichlet spaces with a curvature lower bound, which contain the torus.
The construction and the identities below are the same in both; the reader who prefers the torus may take $\Gamma(u,v)=\nabla u\cdot\nabla v$ throughout.

\emph{The torus.}
$\TT^d=(\RR/2\pi\ZZ)^d$ with the normalised Lebesgue measure, $L_0=-\Delta$ with Fourier symbol $\abs{k}^2$, $\Ecal(u,v)=\int\nabla u\cdot\nabla v$, $\Gamma(u,v)=\nabla u\cdot\nabla v$, and $V=H^1(\TT^d)$.
For a bounded Borel function $\psi$ on $[0,\infty)$, $\psi(L_0)$ is the Fourier multiplier with symbol $\psi(\abs{k}^2)$; the spectral gap on mean-zero functions is $\lambda_1=1$.

\emph{Diffusion spaces.}
We use the framework of Ambrosio, Gigli and Savar\'e \cite{AGS15}: a Dirichlet form $\Ecal$ is the abstract Dirichlet integral, its carr\'e du champ $\Gamma$ the abstract $\nabla u\cdot\nabla v$ with $\Ecal(f,g)=\int\Gamma(f,g)\dint\mu$, strong locality the property that makes $\Gamma$ a derivation (Leibniz rule), and the Bakry--\'Emery condition $\BE(-K,\infty)$ the abstract form of a Ricci lower bound $-K$, used here only through the commutation of $\Gamma$ with the heat flow up to a factor $\me^{2Kt}$ (facts (F3)--(F4) in Section~\ref{subsec:dirichlet-facts}).

\begin{assumption}
\label{ass:M}
$(X,\mathcal B,\mu)$ is a probability space with $\mu$-complete $\sigma$-algebra, and $\Ecal$ is a symmetric, strongly local Dirichlet form on $L^2(X,\mu)$ with dense domain $V:=D(\Ecal)$, generating a Markov semigroup \cite[(2.1), p.~347; p.~348]{AGS15}.
Let $L_0\ge0$ be the self-adjoint operator defined by $\Ecal(f,g)=\int(L_0f)g\dint\mu$ for $f\in D(L_0)$, $g\in V$ (so $L_0=-\Delta_{\Ecal}$ in the notation of \cite[p.~350]{AGS15}), and $P_t=\me^{-tL_0}$.
Moreover:
\begin{enumerate}[label=(\roman*)]
    \item $\mu$ is invariant, $\int P_tf\dint\mu=\int f\dint\mu$ for $f\in L^1(\mu)$ \cite[(2.12), p.~350]{AGS15}; hence $P_t1=1$;
    \item $\sigma(L_0)\subset\{0\}\cup[1,\infty)$ and $\ker L_0=\RR\one$;
    \item $\Ecal$ satisfies $\BE(-K,\infty)$ for some $K\ge0$ in the sense of \cite[Def.~2.4, p.~357]{AGS15} (in the notation of \cite{AGS15}, $\BE(K',\infty)$ with $K'=-K$).
\end{enumerate}
\end{assumption}

The torus satisfies Assumption~\ref{ass:M} with $K=0$ (Corollary~\ref{cor:torus-C2}); further examples and the normalisation of the spectral gap in (ii) are discussed in Section~\ref{subsec:dirichlet-facts}.
From \cite{AGS15} we use two structural facts in this section and two curvature facts in Section~\ref{subsec:dirichlet-facts}.
\begin{enumerate}[label=(F\arabic*)]
    \item Under $\BE(-K,\infty)$, $\Ecal$ admits a carr\'e du champ on all of $V$ \cite[Cor.~2.3, closing sentence, p.~356]{AGS15}: there is a symmetric bilinear map $\Gamma\colon V\times V\to L^1(\mu)$ with $\Gamma(f):=\Gamma(f,f)\ge0$ and $\Ecal(f,g)=\int\Gamma(f,g)\dint\mu$ \cite[(2.6), p.~348; p.~349]{AGS15}. Since $\Gamma(f+\lambda g)\ge0$ a.e.\ for all rational $\lambda$, $\abs{\Gamma(f,g)}^2\le\Gamma(f)\Gamma(g)$ a.e.
    \item $V_\infty:=V\cap L^\infty(\mu)$ is an algebra under pointwise multiplication (a standard fact for Dirichlet forms; see \cite[p.~348]{AGS15} and the references there), and the Leibniz rule $\Gamma(fg,h)=f\Gamma(g,h)+g\Gamma(f,h)$ holds for $f,g,h\in V_\infty$ \cite[display after (2.7), p.~349]{AGS15}.
\end{enumerate}
On the torus both facts are elementary.
Functions are real-valued; the complexification of a real bounded operator has the same norm.

\emph{Notation.}
$\norm{\cdot}$ is the $L^2(\mu)$ norm and $\LL(Y)$ the bounded operators on $Y$.
We write $\norm{u}_1^2=\Ecal(u,u)$ and $\Hd=\{u\in V:\int u\dint\mu=0\}$, a Hilbert space with $\norm{u}\le\norm{u}_1$ by the spectral gap (on $\TT^d$, $\norm{u}_1=\norm{\nabla u}$ and the Poincar\'e constant is $1$ in every dimension).
$\Pi$ is the projection onto mean-zero functions, $R=L_0^{-1}\Pi$ (symbol $\abs{k}^{-2}\one_{k\ne0}$ on $\TT^d$), and $\norm{g}_{-1}:=\norm{L_0^{-1/2}\Pi g}$ is the dual norm, so that $\norm{Rg}_1=\norm{g}_{-1}\le\norm{g}$ and
\begin{equation}
    \label{eq:dual-norm}
    \norm{g}_{-1}=\sup\Bigl\{\Bigl|\int g\varphi\dint\mu\Bigr|:\varphi\in\Hd,\ \norm{\varphi}_1\le1\Bigr\}.
\end{equation}
For a coefficient $a$ with $0<a_{\min}\le a\le a_{\max}$ we always write
\begin{equation}
    \label{eq:intro-constants}
    \bar a=\tfrac12(a_{\min}+a_{\max}),\qquad
    q=\frac{a_{\max}-a_{\min}}{a_{\max}+a_{\min}}=\frac{\kappa-1}{\kappa+1}\in[0,1),\qquad
    \kappa=\frac{a_{\max}}{a_{\min}},\qquad b:=\frac a{\bar a}-1,
\end{equation}
so that $1-q=2/(\kappa+1)$ and, since $a_{\min}/\bar a-1=-q$ and $a_{\max}/\bar a-1=q$, $\norm{b}_\infty\le q$.
For mean-zero $f\in L^2$, $G(a,f)\in\Hd$ denotes the unique solution of $\int a\Gamma(u,v)\dint\mu=\int fv\dint\mu$ for all $v\in\Hd$ (Lax--Milgram; on $\TT^d$ this is the weak form of $-\operatorname{div}(a\nabla u)=f$).

\begin{definition}[Operator networks with functional-calculus layers]
\label{def:class}
Let $\mathcal F$ be a family of bounded Borel functions on $[0,\infty)$.
An \emph{operator network with $\mathcal F$-layers} is a finite composition of maps acting on finitely many channels, each of which is
\begin{enumerate}[label=(\roman*)]
    \item a layer $\psi(L_0)$, $\psi\in\mathcal F$, applied to one channel;
    \item the pointwise product of two channels;
    \item a pointwise affine map of channels, whose coefficients may depend on the constants defining the coefficient class (such as $\bar a$).
\end{enumerate}
We consider two admissible families: $\Fheat=\{\sum_{j\le J}a_j\me^{-t_j\lambda}\}$ (finite exponential sums) and $\Fmix=\Fheat\cup\{s/(1+s\lambda)\}$ (one resolvent allowed).
\end{definition}

\subsection{Normalisation and the preconditioned iteration}
\label{subsec:richardson}
With $b=a/\bar a-1$ the coefficient is a perturbation of size at most $q<1$ of the constant $\bar a$, and the Darcy problem becomes a fixed-point equation for the operator preconditioned by $L_0^{-1}$.

\begin{lemma}[Reduction to a Neumann series]
\label{lem:reduction}
Let $b\in L^\infty(\mu)$ be real, and define $K_b\in\LL(\Hd)$ by $\inner{K_bu,v}_1=\int b\,\Gamma(u,v)\dint\mu$.
Then $\norm{K_b}_{\LL(\Hd)}\le\norm{b}_\infty$.
If $a_{\min}\le a\le a_{\max}$ and $b=a/\bar a-1$, then $\norm{K_b}\le q<1$ and, for mean-zero $f\in L^2$,
\[
    G(a,f)=\bar a^{-1}(I+K_b)^{-1}Rf=\bar a^{-1}\sum_{j\ge0}(-K_b)^jRf\quad\text{in }\Hd,
    \qquad\norm{Rf}_1\le\norm{f}.
\]
\end{lemma}

\begin{proof}
By (F1), $\abs{\int b\Gamma(u,v)}\le\norm{b}_\infty\int\sqrt{\Gamma(u)}\sqrt{\Gamma(v)}\le\norm{b}_\infty\norm{u}_1\norm{v}_1$.
For $u,v\in\Hd$, $\int a\Gamma(u,v)\dint\mu=\bar a\bigl(\inner{u,v}_1+\inner{K_bu,v}_1\bigr)$ and $\int fv\dint\mu=\inner{Rf,v}_1$, so the weak formulation reads $\bar a(I+K_b)u=Rf$; since $\norm{K_b}\le\norm{b}_\infty\le q$ the Neumann series converges in $\LL(\Hd)$.
Finally $\norm{Rf}_1=\norm{f}_{-1}\le\norm{f}$.
\end{proof}

The truncated series $\bar a^{-1}\sum_{j\le N}(-K_b)^jRf$ is the $N$-th iterate of
\begin{equation}
    \label{eq:exact-richardson}
    v_0=Rf,\qquad v_{i+1}=Rf-K_bv_i,
\end{equation}
the Richardson iteration for $\bar a^{-1}L_0^{-1}L_a$ with the midpoint normalisation $\bar a$, whose contraction factor is $q$.
On $\TT^d$, $K_b=L_0^{-1}L_b$ with $L_bu=-\operatorname{div}(b\nabla u)$, and \eqref{eq:exact-richardson} is the Laplacian-preconditioned iteration of \cite{CG73,NHT09,GMNS19}, unrolled in \cite{KLM21,MS24}; the boundedness of $K_b$ on $\Hd$ is an instance of the theory of equivalent operators \cite{FMP90}.
What is new here is the realisation of one step inside the primitive set of Definition~\ref{def:class} and its control in the operator norm of $\Hd$.

\subsection{The preconditioned operator: multiplication plus smoothing}
\label{subsec:representation}
The operator $K_b$ is not a pointwise operation, but it is one up to a smoothing term.
This identity dictates the shape of the block.

\begin{theorem}[Representation]
\label{thm:representation}
Let $b\in L^\infty(\mu)\cap D(L_0)$ be real with $\norm{b}_{\mathcal B}:=\norm{\sqrt{\Gamma(b)}}_\infty+\norm{L_0b}_\infty<\infty$.
Then, as an identity in $\LL(\Hd)$,
\begin{equation}
    \label{eq:representation}
    K_b=\Pi M_b+R\bigl(\Gamma(b,\cdot)-M_{L_0b}\bigr),
    \qquad
    \norm{\Gamma(b,u)-uL_0b}\le\norm{b}_{\mathcal B}\norm{u}_1\quad(u\in\Hd),
\end{equation}
where $M_g$ denotes multiplication by $g$.
In particular the remainder $K_b-\Pi M_b$ maps $\Hd$ boundedly into $D(L_0)$, with $\norm{L_0(K_b-\Pi M_b)u}\le\norm{b}_{\mathcal B}\norm{u}_1$.
\end{theorem}

\begin{proof}
For $u,v\in V_\infty$, two applications of the Leibniz rule (F2) and the definition of $L_0$ give
\[
    \int b\,\Gamma(u,v)=\Ecal(bu,v)-\int u\,\Gamma(b,v)=\Ecal(bu,v)-\int(L_0b)uv+\int v\,\Gamma(b,u).
\]
Here $\Ecal(bu,v)=\inner{\Pi(bu),v}_1$ and $\int gv=\inner{Rg,v}_1$ for $v\in\Hd$, which gives \eqref{eq:representation} for $u,v\in\Hd\cap V_\infty$.
All terms are continuous in $u,v$ for the norm $\Ecal_1(u)^{1/2}=(\Ecal(u,u)+\norm{u}^2)^{1/2}$, since $\norm{bu}_1\le\norm{b}_\infty\norm{u}_1+\norm{\sqrt{\Gamma(b)}}_\infty\norm{u}$ by (F2) and (F1), and Lemma~\ref{lem:density} (density of $V_\infty$ in $V$) extends the identity to $\Hd$.
The bound on $\Gamma(b,u)-uL_0b$ follows from (F1) and $\norm{u}\le\norm{u}_1$; the last claim from $L_0Rg=\Pi g$.
\end{proof}

What the identity buys: the multiplication part $\Pi M_b$ is local and the architecture applies it \emph{exactly}, which is why the contraction factor $q=\norm{b}_\infty$ is inherited without loss; only the remainder has to be approximated, and since it is smoothing, the inverse $L_0^{-1}$ is needed only as a map from $L^2$ to $\Hd$ that is accurate relative to $L_0^{-1}$ on a spectral band (Lemma~\ref{lem:HR} below).
The remainder contains the coupling $\Gamma(b,u)$, which the primitives cannot produce directly; Section~\ref{subsec:gamma} shows how a heat layer produces it.

\begin{remark}[Fourier form; equivalent operators]
\label{rem:R-classical}
On $\TT^d$, testing with $v=\me^{ik'\cdot x}$ and writing $k'=k+\mu$ for the output frequency $k'$, the frequency $k$ of $u$ and the frequency $\mu$ of $b$ gives
\[
    \widehat{K_bu}_{k'}=\sum_\mu\hat b_\mu\,\frac{k\cdot k'}{\abs{k'}^2}\,\hat u_k\qquad(k'\ne0),
\]
and \eqref{eq:representation} is the splitting $k\cdot k'=\abs{k'}^2-\mu\cdot k-\abs{\mu}^2$ of the numerator: the three terms are $\Pi M_b$, $R\Gamma(b,\cdot)$ and $-RM_{L_0b}$.
Proposition~\ref{prop:fourier-representation} extends this form to $b\in A^1$ (one derivative in the Wiener sense), where $L_0b$ need not be a function.
The identity is the classical structure in the form we need and is not claimed as new: in the terminology of Axelsson and Kar\'atson \cite{AK04,AK07,AK09}, $L_0$ and $L_a$ are equivalent but not compact-equivalent, since $I+K_b=\bar a^{-1}L_0^{-1}L_a$ is a compact perturbation of the multiplication-type operator $\Pi M_{a/\bar a}$ on $\Hd$ rather than of a multiple of the identity; accordingly the spectrum of $L_0^{-1}L_a$ contains the range of $a$ \cite{NHT09}, and the eigenvalues of its Galerkin discretisations are localised near the values of $a$ \cite{GMNS19}.
Only the explicit $\Gamma$ form of the remainder is specific to this paper.
\end{remark}

\subsection{The product defect synthesises the coupling}
\label{subsec:gamma}
The remainder in \eqref{eq:representation} contains $\Gamma(b,u)=\nabla b\cdot\nabla u$, which no function of $L_0$ applied to $b$ or to $u$ separately can produce.
One heat layer applied to a \emph{product} can.
For $t>0$ define the product defect
\begin{equation}
    \label{eq:gamma-t}
    \Gamma_t(b,u):=\frac1{2t}\bigl[P_t(bu)-P_tb\,P_tu\bigr].
\end{equation}

\begin{lemma}[Interpolation identity]
\label{lem:interpolation}
Let $b\in L^\infty\cap D(L_0)$ with $L_0b\in L^\infty$ and $\sqrt{\Gamma(b)}\in L^\infty$, and let $u\in V_\infty$.
Then, in $L^2$,
\[
    P_t(bu)-P_tb\,P_tu=2\int_0^tP_s\Gamma(P_{t-s}b,P_{t-s}u)\dint s .
\]
\end{lemma}

The identity is the semigroup interpolation formula \cite{BGL14}; it is proved under the present hypotheses in Appendix~\ref{app:aux}, where the gradient bound of $\BE(-K,\infty)$ provides the integrability.
By the lemma, $\Gamma_t(b,u)$ is the average over $s\in[0,t]$ of $P_s\Gamma(P_{t-s}b,P_{t-s}u)$, which tends to $\Gamma(b,u)$ as $t\to0$.
The quantitative statement, and the only place where $t$ enters the analysis, is the $\Gamma$-consistency estimate $\norm{(\Gamma_t-\Gamma)(b,u)}_{-1}\le1.37\,\me^{Kt}\sqrt t\,\norm{b}_{\mathcal B}\norm{u}_1$ of Theorem~\ref{thm:gamma-consistency}; the dual norm is essential, since it is what allows $\norm{L_0b}_\infty$ rather than a bound on $L_0^2b$.
On the torus the defect has an explicit Fourier multiplier (Lemma~\ref{lem:product-defect}), and on Wiener classes $A^{1+\theta}$ the rate becomes $t^{\theta/2}$, which the product-defect term attains (Section~\ref{sec:torus}).

\subsection{Functional-calculus approximate inverses}
\label{subsec:inverse}
$L_0^{-1}$ is not a heat layer.
We use $R_m=r(L_0)$ for a scalar function $r$, and the only property of $r$ that the analysis needs is relative accuracy on a spectral band together with boundedness; because $R_m$ acts only on the smoothing remainder of \eqref{eq:representation}, the band is paid for once, by a tail term.

\begin{definition}[Band hypothesis]
\label{def:HR}
Let $\delta\in(0,1]$ and $\Lambda\in[1,\infty]$.
A function $r\colon[0,\infty)\to\RR$ with $r(0)=0$ satisfies \HR{} if
\[
    \abs{\lambda r(\lambda)-1}\le\delta\ \text{ for }\lambda\in[1,\Lambda],
    \qquad
    0\le\lambda r(\lambda)\le1+\delta\ \text{ for }\lambda\in[1,\infty).
\]
For $\Lambda=\infty$ the first condition is required on $[1,\infty)$.
We then write $R_m:=r(L_0)$.
\end{definition}

Since $r(0)=0$, $R_m=R_m\Pi$ annihilates constants, and only the values of $r$ on $\sigma(L_0)\setminus\{0\}\subset[1,\infty)$ matter.

\begin{lemma}[Consequences of \HR]
\label{lem:HR}
If $r$ satisfies \HR{}, then for $\lambda\ge1$
\[
    \abs{r(\lambda)-\lambda^{-1}}\le\lambda^{-1}\bigl(\delta+2\cdot\one_{\lambda>\Lambda}\bigr),\qquad
    \abs{r(\lambda)}\le(1+\delta)\lambda^{-1},
\]
and for every $g\in L^2$
\begin{equation}
    \label{eq:HR-consequences}
    \norm{(R_m-R)g}_1\le(\delta+2\Lambda^{-1/2})\norm{g},\qquad\norm{R_mg}_1\le(1+\delta)\norm{g}_{-1}\le(1+\delta)\norm{g},
\end{equation}
with $\Lambda^{-1/2}:=0$ if $\Lambda=\infty$.
\end{lemma}

\begin{proof}
On $[1,\Lambda]$, $\abs{r-\lambda^{-1}}=\lambda^{-1}\abs{\lambda r-1}\le\delta\lambda^{-1}$; for $\lambda>\Lambda$, $\lambda r\in[0,1+\delta]$ gives $\abs{\lambda r-1}\le1\le\delta+2$; the bound on $\abs r$ is immediate.
By the spectral theorem, $\norm{(R_m-R)g}_1^2=\int_{\lambda\ge1}\lambda\abs{r(\lambda)-\lambda^{-1}}^2\dint\inner{E_\lambda g,g}\le\sup_{\lambda\ge1}\lambda^{-1}(\delta+2\cdot\one_{\lambda>\Lambda})^2\norm{g}^2\le(\delta+2\Lambda^{-1/2})^2\norm{g}^2$,
and $\norm{R_mg}_1^2=\int_{\lambda\ge1}(\lambda r(\lambda))^2\lambda^{-1}\dint\inner{E_\lambda g,g}\le(1+\delta)^2\norm{g}_{-1}^2$, because $r(0)=0$.
\end{proof}

Two realisations are used.

\begin{lemma}[Sinc-quadrature approximate inverse]
\label{lem:sinc}
Let $\delta_0\in(0,1]$ and $\Lambda\ge1$.
Put
\[
    h=\frac{\pi^2}{\pi^2+2\log(1/\delta_0)},\quad
    M=\Bigl\lceil\frac{\log(8\Lambda/\delta_0)}{h}\Bigr\rceil-1,\quad
    \Nsinc=\Bigl\lceil\frac{\log\bigl(2\log\frac{8}{h\delta_0}\bigr)}{h}\Bigr\rceil-1,\quad
    T=\log\frac5{\delta_0},
\]
and
\[
    r(\lambda)=\bigl(1-\me^{-T\lambda}\bigr)\,h\sum_{j=-M}^{\Nsinc}\me^{jh}\me^{-\lambda\me^{jh}}.
\]
Then $r(0)=0$ and $r$ satisfies \HR{} on $[1,\Lambda]$ with $\delta=\delta_0$.
The operator $R_m=r(L_0)=(I-P_T)\,h\sum_{j=-M}^{\Nsinc}\me^{jh}P_{\me^{jh}}$ is a composition of $\Fheat$-layers using $m:=M+\Nsinc+1$ heat times $\me^{jh}$ and the further time $T$, and
\[
    m=O\Bigl(\log\frac1{\delta_0}\,\log\frac{\Lambda}{\delta_0}\Bigr).
\]
\end{lemma}

The proof, by Poisson summation, is given in Appendix~\ref{app:sinc}, and an explicit bound for $m$ in Proposition~\ref{prop:counts}.
Exponential-sum and sinc approximations of $1/\lambda$ are classical \cite{BH05,BH26}; Lemma~\ref{lem:sinc} is included only for its explicit constants in the form \HR{}.

\begin{lemma}[Resolvent approximate inverse]
\label{lem:resolvent}
Let $s,T>0$.
The function $r(\lambda)=(1-\me^{-T\lambda})\,s/(1+s\lambda)$ satisfies $r(0)=0$ and \HR{} with $\Lambda=\infty$ and $\delta=\me^{-T}+1/(1+s)$, provided this quantity is at most $1$.
\end{lemma}

\begin{proof}
For $\lambda\ge1$,
$\lambda r(\lambda)=(1-\me^{-T\lambda})\bigl(1-\frac1{1+s\lambda}\bigr)\in\bigl[(1-\me^{-T})(1-\frac1{1+s}),1\bigr)$,
and $(1-\me^{-T})(1-\frac1{1+s})\ge1-\me^{-T}-\frac1{1+s}$.
\end{proof}

\subsection{The block and its five error terms}
\label{subsec:block}
Replacing each ingredient of \eqref{eq:representation} by a primitive gives the block
\begin{equation}
    \label{eq:block}
    \Knet u:=(I-P_T)(bu)+R_m\bigl(\Gamma_t(b,u)-u\,b_\tau\bigr),\qquad b_\tau:=(I-P_\tau)b/\tau,
\end{equation}
with $I-P_T$ for $\Pi$, $\Gamma_t$ for $\Gamma$, $b_\tau$ for $L_0b$, and $R_m$ for $R$; since $I-P_T$ and $R_m$ annihilate constants, the output has mean zero, which is needed because $P_tb\,P_tu$ and $u\,b_\tau$ have non-zero mean in general.

\begin{lemma}[Five-term decomposition]
\label{lem:five-terms}
Let $b$ be as in Theorem~\ref{thm:representation}, $t,\tau,T>0$, and $R_m=r(L_0)$ with $r$ satisfying \HR{}.
Then $\Knet\in\LL(\Hd)$ has mean-zero output, and
\begin{equation}
    \label{eq:five-terms}
    \Knet-K_b=\underbrace{-P_T\Pi M_b}_{(1)}+\underbrace{(R_m-R)\Gamma(b,\cdot)}_{(2)}+\underbrace{R_m(\Gamma_t-\Gamma)(b,\cdot)}_{(3)}
    -\underbrace{(R_m-R)M_{L_0b}}_{(4)}-\underbrace{R_mM_{b_\tau-L_0b}}_{(5)} .
\end{equation}
\end{lemma}

\begin{proof}
Since $P_T1=1$, $(I-P_T)M_b=(I-P_T)\Pi M_b=\Pi M_b-P_T\Pi M_b$; subtracting \eqref{eq:representation} and adding and subtracting $R_m\Gamma(b,\cdot)$ and $R_mM_{L_0b}$ gives \eqref{eq:five-terms}.
For boundedness, $\norm{bu}_1\le\norm{b}_\infty\norm{u}_1+\norm{\sqrt{\Gamma(b)}}_\infty\norm{u}$ by (F1)--(F2), $\norm{\Gamma_t(b,u)}\le\norm{b}_\infty\norm{u}/t$, $\norm{b_\tau}_\infty\le\norm{L_0b}_\infty$ because $b-P_\tau b=\int_0^\tau P_\sigma L_0b\dint\sigma$, and $R_m\colon L^2\to\Hd$ is bounded by \eqref{eq:HR-consequences}.
\end{proof}

Each term is one replacement, controlled by one parameter:
\begin{center}
\small
\renewcommand{\arraystretch}{1.2}
\begin{tabular}{clll}
\toprule
term & replacement & parameter & what controls it \\
\midrule
(1) & $\Pi\to I-P_T$ & $T$ & $\norm{P_T\Pi}_{L^2\to\Hd}\le\me^{-T}$ (Lemma~\ref{lem:spectral}) \\
(2) & $R\to R_m$ on $\Gamma(b,\cdot)$ & $\delta,\Lambda$ & band hypothesis, Lemma~\ref{lem:HR} \\
(3) & $\Gamma\to\Gamma_t$ & $t$ & $\Gamma$-consistency, Theorem~\ref{thm:gamma-consistency} / Lemma~\ref{lem:product-defect} \\
(4) & $R\to R_m$ on $M_{L_0b}$ & $\delta,\Lambda$ & band hypothesis, Lemma~\ref{lem:HR} \\
(5) & $L_0b\to b_\tau$ & $\tau$ & $\norm{b_\tau-L_0b}$ in the dual norm \\
\bottomrule
\end{tabular}
\end{center}
Term (3) is the mechanism; (2) and (4) are the price of not having $L_0^{-1}$; (1) and (5) the price of not having $\Pi$ and $L_0b$.
Theorems~\ref{thm:dirichlet-step} and~\ref{thm:step} bound the five terms in $\LL(\Hd)$ in the two settings; the estimates differ, the structure does not.

\subsection{Propagating the step error through the contraction}
\label{subsec:budget}
Fix the accuracy $\varepsilon\in(0,1)$ and set
\begin{equation}
    \label{eq:eta-N}
    \eta_0:=\min\Bigl\{\frac{\varepsilon\bar a(1-q)^2}{8},\ \frac{1-q}2\Bigr\},\qquad
    N:=\max\Bigl\{0,\Bigl\lceil\frac{\log\frac4{\varepsilon\bar a(1-q)}}{\log(1/q)}\Bigr\rceil\Bigr\}\quad(N=0\text{ if }q=0).
\end{equation}
Since $\bar a(1-q)=a_{\min}$, equivalently $N=\max\{0,\lceil\log(4/(\varepsilon a_{\min}))/\log(1/q)\rceil\}$ for $q>0$ and $\eta_0=\min\{\varepsilon a_{\min}(1-q)/8,(1-q)/2\}$; for large $\kappa$, $\log(1/q)\sim2/\kappa$ and the depth grows like $(\kappa/2)\log(4/(\varepsilon a_{\min}))$.

\begin{proposition}[Inexact preconditioned Richardson iteration]
\label{prop:budget}
Let $a_{\min}\le a\le a_{\max}$, $b=a/\bar a-1$ and $K_b$ as in Lemma~\ref{lem:reduction}.
Let $R_m=r(L_0)$ with $r(0)=0$ and $\norm{(R_m-R)f}_1\le\delta'\norm{f}$ for all $f\in L^2$, and let $\Knet\in\LL(\Hd)$ with $\norm{\Knet-K_b}_{\LL(\Hd)}\le\eta\le(1-q)/2$.
Then the scheme
\begin{equation}
    \label{eq:scheme}
    v_0=R_mf,\qquad v_{i+1}=R_mf-\Knet v_i\quad(0\le i<N),\qquad N_\varepsilon(a,f)=\bar a^{-1}v_N
\end{equation}
satisfies, for every mean-zero $f\in L^2$ and every $N\ge0$,
\begin{equation}
    \label{eq:budget}
    \bar a\,\norm{G(a,f)-N_\varepsilon(a,f)}_1\le\Bigl[\frac{q^{N+1}}{1-q}+\frac{4\eta}{(1-q)^2}+\frac{2\delta'}{1-q}\Bigr]\norm{f}.
\end{equation}
In particular, if $\eta\le\eta_0$, $\delta'\le\varepsilon\bar a(1-q)/8$ and $N$ is given by \eqref{eq:eta-N}, then $\norm{G(a,f)-N_\varepsilon(a,f)}_1\le\varepsilon\norm{f}$.
\end{proposition}

\begin{proof}
Since $r(0)=0$, $v_0\in\Hd$, and $\Knet\in\LL(\Hd)$ keeps the iterates in $\Hd$; by induction $v_N=\sum_{j=0}^N(-\Knet)^jR_mf$, while $G(a,f)=\bar a^{-1}\sum_{j\ge0}(-K_b)^jRf$ by Lemma~\ref{lem:reduction}.
Put $q':=q+\eta\le\frac{1+q}2$, so that $\norm{\Knet}\le q'$ and $1-q'\ge(1-q)/2$.
Then
\begin{multline*}
    \bar a\,\norm{G-N_\varepsilon}_1\le\Bigl\|\sum_{j>N}(-K_b)^jRf\Bigr\|_1+\Bigl\|\sum_{j\le N}\bigl[(-K_b)^j-(-\Knet)^j\bigr]Rf\Bigr\|_1\\
    +\Bigl\|\sum_{j\le N}(-\Knet)^j(R-R_m)f\Bigr\|_1 .
\end{multline*}
The first term is at most $q^{N+1}(1-q)^{-1}\norm{f}$.
For the second, telescoping gives $\norm{K_b^j-(\Knet)^j}\le j\eta\,q'^{\,j-1}$, so it is at most $\eta(1-q')^{-2}\norm{f}\le4\eta(1-q)^{-2}\norm{f}$.
For the third, $\sum_jq'^{\,j}\le(1-q')^{-1}\le2(1-q)^{-1}$ gives the bound $2\delta'(1-q)^{-1}\norm{f}$.
This proves \eqref{eq:budget}.

For the last claim, let $\norm f\le1$.
By the choice of $N$, $q^N\le\varepsilon\bar a(1-q)/4$ (for $N=0$ and $q>0$ the definition gives $\varepsilon\bar a(1-q)\ge4$; for $q=0$ the first term vanishes), so the first term is at most $\varepsilon\bar a/4$.
The second is at most $\varepsilon\bar a/2$ because $\eta\le\eta_0\le\varepsilon\bar a(1-q)^2/8$, and the third is at most $\varepsilon\bar a/4$ because $\delta'\le\varepsilon\bar a(1-q)/8$.
\end{proof}

Proposition~\ref{prop:budget} is used with $\eta=\eta_0$ in Theorem~\ref{thm:dirichlet-main} (Section~\ref{sec:dirichlet}) and in Theorem~\ref{thm:main} (Section~\ref{sec:torus}).
The two parameter recipes that make the five terms of \eqref{eq:five-terms} sum to at most $\eta_0$ are collected in Table~\ref{tab:recipes}; for both, the inverse error $\delta'=\delta+2\Lambda^{-1/2}$ of Lemma~\ref{lem:HR} is at most $\eta_0/(5B)\le\varepsilon\bar a(1-q)/8$, since $\eta_0\le\varepsilon\bar a(1-q)^2/8$ and $B\ge1$.

\begin{table}[ht]
    \centering
    \caption{The two parameter recipes. Common to both: $\eta_0$ and $N$ from \eqref{eq:eta-N}; in the heat family $R_m$ is the sinc-quadrature inverse of Lemma~\ref{lem:sinc} at $(\delta_0,\Lambda)$, in the mixed family $R_m=(I-P_T)\,sJ_s$ with $J_s=(I+sL_0)^{-1}$; the same $T$ is used in $R_m$ and in the factor $I-P_T$ of \eqref{eq:block}. All parameters depend on the class constants and $\varepsilon$ only.}
    \label{tab:recipes}
    \small
    \renewcommand{\arraystretch}{1.2}
    \begin{tabular}{lll}
        \toprule
        & recipe C$^2$ (Theorem~\ref{thm:dirichlet-main}) & recipe W (Theorem~\ref{thm:main}) \\
        & class $\AB(B)$, curvature $K\ge0$ & class $\Acl$, $\theta\in(0,1]$, $\TT^d$ \\
        \midrule
        $\delta_0$ & $\eta_0/(10B)$ & $\eta_0/(15B)$ \\
        $t$ & $t_0\me^{-2Kt_0}$, $t_0=(\eta_0/5B)^2$ & $\frac1{40}(\eta_0/20B)^{2/\theta}$ \\
        $\tau$ & $\tau_0\me^{-2K\tau_0}$, $\tau_0=(\eta_0/7B)^2$ & $2(\eta_0/20B)^{2/\theta}$ \\
        heat family: $\Lambda$, $T$ & $(20B/\eta_0)^2$, $\log(5/\delta_0)$ & $(60B/\eta_0)^{2/\theta}$, $\log(5/\delta_0)$ \\
        mixed family: $s$, $T$ & $2/\delta_0=20B/\eta_0$, $\log(2/\delta_0)$ & $2/\delta_0=30B/\eta_0$, $\log(2/\delta_0)$ \\
        step error $\norm{\Knet-K_b}_{\LL(\Hd)}\le$ & $0.82\,\eta_0$ (heat), $0.75\,\eta_0$ (mixed) & $\eta_0$ (both) \\
        \bottomrule
    \end{tabular}
\end{table}

\subsection{The algorithm}
\label{subsec:algorithm}
By linearity, $\Knet v=(I-P_T)\bigl[bv+S(\Gamma_t(b,v)-v\,b_\tau)\bigr]$ with $R_m=(I-P_T)S$, where $S=h\sum_j\me^{jh}P_{\me^{jh}}$ (heat family) or $S=sJ_s$ (mixed family).
After the precomputation of $b=a/\bar a-1$, $P_tb$, $b_\tau=(b-P_\tau b)/\tau$ and $v_0=(I-P_T)Sf$, block $i$ forms the products $w=bv_i$, $P_tb\cdot P_tv_i$ and $v_ib_\tau$, and sets
$g=\frac1{2t}[P_tw-P_tb\,P_tv_i]-v_ib_\tau$ and $v_{i+1}=v_0-(I-P_T)[w+Sg]$; all blocks share their parameters, and each costs two heat applications at time $t$, one application of $S$, one at time $T$, and three products (Proposition~\ref{prop:counts}).
The division by $2t$ in $g$ is where the derivative is extracted from the semigroup; it is the source of the cancellation discussed in Section~\ref{subsec:floating}.

\section{Dirichlet spaces with a curvature lower bound: the main theorem}
\label{sec:dirichlet}

This section proves Theorem~B and derives Theorem~A from it.
What remains after Section~\ref{sec:mechanism} is to bound the five terms of \eqref{eq:five-terms} in $\LL(\Hd)$; the only terms that need more than the Leibniz rule and the spectral gap are the product-defect term (3) and the related term (5), and both are controlled by how the heat flow interacts with $\Gamma$, which is what a curvature lower bound quantifies.

\subsection{Which hypotheses carry the estimates}
\label{subsec:dirichlet-facts}
Three hypotheses of Assumption~\ref{ass:M} do the work, and each is used in one place.

\emph{The spectral gap} (ii) gives $\norm{Rf}_1\le\norm{f}$ (Lemma~\ref{lem:reduction}), the Poincar\'e inequality $\norm{u}\le\norm{u}_1$, and the meaning of the band hypothesis \HR{} on $[1,\Lambda]$.
Its normalisation $\lambda_1\ge1$ is a convention: a form whose generator has spectral gap $\lambda_1\in(0,1)$ is reduced to this case by replacing $L_0$ by $L_0/\lambda_1$; then $\Gamma$ becomes $\Gamma/\lambda_1$, so that $\norm{\sqrt{\Gamma(b)}}_\infty$ is divided by $\lambda_1^{1/2}$ and $\norm{L_0b}_\infty$ by $\lambda_1$ in the class constant $B$, $K$ becomes $K/\lambda_1$, and a heat application $\me^{-tL_0/\lambda_1}$ is the original semigroup at time $t/\lambda_1$ ($s$ becomes $s/\lambda_1$).
This is the only way the spectrum of $L_0$ enters.

\emph{Strong locality}, through the Leibniz rule (F2), gives the representation identity (Theorem~\ref{thm:representation}) and the $\dot H^{-1}$ bound of Lemma~\ref{lem:Hminus} below, which is what lets the product-defect estimate use $\norm{L_0b}_\infty$ instead of a bound on $L_0^2b$.
Graph Laplacians have a carr\'e du champ but no Leibniz rule, and are therefore not covered.

\emph{The curvature bound} $\BE(-K,\infty)$ (iii) enters only through the following two consequences, both from \cite{AGS15}, and only in the proofs of Theorem~\ref{thm:gamma-consistency}, of Lemma~\ref{lem:interpolation} and of the integrated gradient bound (Lemma~\ref{lem:bochner}), on which term (5) of Theorem~\ref{thm:dirichlet-step} rests.
\begin{enumerate}[label=(F\arabic*),start=3]
    \item Gradient bound \cite[Cor.~2.3(vi), (2.35) with $\nu=0$, p.~356]{AGS15}: $\Gamma(P_rf)\le\me^{2Kr}P_r\Gamma(f)$ a.e.\ for $f\in V$, $r>0$.
    \item Reverse local Poincar\'e inequality \cite[Cor.~2.3(iv), (2.34) with $\nu=0$, p.~356]{AGS15}: for $f\in L^2$ and $r>0$,
    \[
        \frac{1-\me^{-2Kr}}{2K}\,\Gamma(P_rf)\le\tfrac12P_r(f^2)-\tfrac12(P_rf)^2
    \]
    (read $r$ for the fraction if $K=0$). For $f\in L^\infty$ this gives, with $1-\me^{-y}\ge y\me^{-y}$,
    \begin{equation}
        \label{eq:reverse-poincare}
        \norm{\sqrt{\Gamma(P_rf)}}_\infty\le\me^{Kr}(2r)^{-1/2}\norm{f}_\infty .
    \end{equation}
\end{enumerate}
(F3) says that the heat flow does not increase $\Gamma$ by more than $\me^{2Kr}$, and (F4) that one heat step of length $r$ regularises a bounded function to one with $\sqrt{\Gamma}$ of size $r^{-1/2}$; these are the two facts behind the rate $\sqrt t$ of the product defect, and since the curvature appears only through factors $\me^{Kr}$ with $r\le t$ or $r\le\tau$, it enters the final parameters only through $\me^{-2Kt_0}$, $\me^{-2K\tau_0}$ in Table~\ref{tab:recipes}.

\emph{Why nothing depends on the dimension.}
None of the facts (F1)--(F4), and none of the estimates below, involves a dimension, a volume growth, or an eigenvalue of $L_0$ other than through $\lambda_1\ge1$; the constants $c_0$, $1.37$ and $1.89$ that appear are numerical.
On the torus the same proof therefore gives constants independent of $d$, in contrast with a Sobolev-embedding argument, where $d$ enters through the embedding constant.

\emph{Examples.}
The flat torus with $L_0=-\Delta$ satisfies Assumption~\ref{ass:M} with $K=0$ (Corollary~\ref{cor:torus-C2}).
Compact connected Riemannian manifolds with $L_0=-\Delta_g/\lambda_1$ and Ricci curvature bounded below by $-K\lambda_1$, and weighted manifolds with the corresponding Bakry--\'Emery tensor bound, are the standard examples \cite[Section~1]{AGS15}, \cite[Section~1.16 and Appendix~C]{BGL14}; the metric measure spaces with Riemannian Ricci curvature bounded below by $-K$ of \cite{AGS14} satisfy $\BE(-K,\infty)$ \cite[Theorem~6.2]{AGS14} (see \cite[p.~344]{AGS15}) and fall under Assumption~\ref{ass:M} whenever the reference measure is finite (normalised to a probability measure) and (ii) holds.
None of these examples is used in the proofs.

\emph{The coefficient class.}
For $0<a_{\min}\le a_{\max}$ and $B\ge1$, with $\bar a$, $q$, $\kappa$ as in \eqref{eq:intro-constants},
\begin{equation}
    \label{eq:class-B}
    \AB(B)=\Bigl\{a\ \Big|\ a_{\min}\le a\le a_{\max}\ \mu\text{-a.e.},\ b:=\frac a{\bar a}-1\in D(L_0),\
    \norm{b}_{\mathcal B}=\norm{\sqrt{\Gamma(b)}}_\infty+\norm{L_0b}_\infty\le B\Bigr\}.
\end{equation}
For $a\in\AB(B)$ one has $\norm{b}_\infty\le q$, so $b\in V_\infty\cap D(L_0)$ and Theorem~\ref{thm:representation} and Lemma~\ref{lem:five-terms} apply.

\subsection{Statements}
\label{subsec:dirichlet-statements}
The first theorem is the quantitative form of the mechanism; it is the only place where the heat time $t$ of the product defect enters.

\begin{theorem}[$\Gamma$-consistency]
\label{thm:gamma-consistency}
Let Assumption~\ref{ass:M} hold, and let $b\in L^\infty(\mu)\cap D(L_0)$ be real with $\norm{b}_{\mathcal B}<\infty$.
Then for all $u\in V$ and $t>0$,
\[
    \norm{(\Gamma_t-\Gamma)(b,u)}_{-1}\le1.37\,\me^{Kt}\sqrt t\,\norm{b}_{\mathcal B}\norm{u}_1 .
\]
\end{theorem}

The second theorem bounds the five terms of \eqref{eq:five-terms}; term (3) is Theorem~\ref{thm:gamma-consistency} composed with $R_m$, and term (5) is its analogue for the channel $b_\tau$.

\begin{theorem}[Inexact preconditioned step]
\label{thm:dirichlet-step}
Let Assumption~\ref{ass:M} hold, let $b\in L^\infty(\mu)\cap D(L_0)$ be real with $\norm{b}_{\mathcal B}<\infty$, let $t,\tau>0$, $T\ge\frac12$, and let $R_m=r(L_0)$ with $r(0)=0$ satisfying \HR{} with $(\delta,\Lambda)$ (Definition~\ref{def:HR}; in particular $\delta\le1$).
Then $\Knet\in\LL(\Hd)$ has mean-zero output and
\begin{multline*}
    \norm{\Knet-K_b}_{\LL(\Hd)}\le\me^{-T}\norm{b}_\infty+(\delta+2\Lambda^{-1/2})\norm{b}_{\mathcal B}\\
    +(1+\delta)\bigl(1.37\,\me^{Kt}\sqrt t\,\norm{b}_{\mathcal B}+1.89\,\me^{K\tau}\sqrt\tau\,\norm{L_0b}_\infty\bigr).
\end{multline*}
\end{theorem}

The main theorem inserts the recipe of Table~\ref{tab:recipes} into Theorem~\ref{thm:dirichlet-step} and Proposition~\ref{prop:budget}.

\begin{theorem}[Emulation on Dirichlet spaces]
\label{thm:dirichlet-main}
Let Assumption~\ref{ass:M} hold, $0<a_{\min}\le a_{\max}$, $B\ge1$ and $\varepsilon\in(0,1)$; let $\eta_0$ and $N$ be given by \eqref{eq:eta-N}, and let the block \eqref{eq:block} use recipe~C$^2$ of Table~\ref{tab:recipes}, in the heat family (a) or in the mixed family (b).
Then $\norm{\Knet-K_b}_{\LL(\Hd)}\le0.82\,\eta_0$ in case (a) and $\le0.75\,\eta_0$ in case (b) for every $a\in\AB(B)$, and the scheme \eqref{eq:scheme} satisfies
\[
    \sup_{a\in\AB(B)}\ \sup_{\norm{f}_{L^2}\le1,\ \int f=0}\ \norm{G(a,f)-N_\varepsilon(a,f)}_{\Hd}\le\varepsilon .
\]
All parameters depend only on $B$, $a_{\min}$, $a_{\max}$, $\varepsilon$ and $K$, not on $X$, $\mu$, a dimension, or the spectrum of $L_0$ beyond $\lambda_1\ge1$.
The operator counts of Proposition~\ref{prop:counts} hold with $m$ given by Lemma~\ref{lem:sinc} at $(\delta_0,\Lambda)$.
\end{theorem}

\begin{corollary}[$\TT^d$ with a $C^2$-type class; Theorem~A]
\label{cor:torus-C2}
On $\TT^d$ with $L_0=-\Delta$, Theorem~\ref{thm:dirichlet-main} holds with $K=0$ for the class of $a$ with $a_{\min}\le a\le a_{\max}$ and $b=a/\bar a-1\in W^{2,\infty}(\TT^d)$, $\norm{\nabla b}_\infty+\norm{\Delta b}_\infty\le B$; the parameters do not depend on $d$.
This class contains every $a$ with $a_{\min}\le a\le a_{\max}$ and $\norm{b}_{A^2}\le B/2$ in the Wiener norm \eqref{eq:wiener}, and, if $a_{\min}<a_{\max}$, it is not contained in any ball of $A^2(\TT^d)$.
\end{corollary}

\subsection{Proofs}
\label{subsec:dirichlet-proofs}
Throughout, Assumption~\ref{ass:M} holds, and we write $\alpha=\norm{\sqrt{\Gamma(b)}}_\infty$ and $\beta=\norm{L_0b}_\infty$, so that $\norm{b}_{\mathcal B}=\alpha+\beta$.
Two auxiliary lemmas are used repeatedly; the first is pure spectral calculus and is proved in Appendix~\ref{app:aux}, the second is the place where strong locality enters the quantitative estimates.

\begin{lemma}[Spectral facts]
\label{lem:spectral}
Let $c_0:=\sup_{x>0}(1-\me^{-x^2})/x\le0.6382$.
For $u\in V$, $g\in L^2$ and $r,s>0$:
\begin{gather*}
    \norm{(P_r-I)u}\le c_0\sqrt r\,\norm{u}_1;\qquad
    \norm{(P_s-I)g}_{-1}\le c_0\sqrt s\,\norm{g};\\
    \norm{P_ru}_1\le\norm{u}_1;\qquad
    \norm{P_T\Pi g}_1\le\me^{-T}\norm{g}\ \text{ for }T\ge\tfrac12 .
\end{gather*}
\end{lemma}

\begin{lemma}[$\dot H^{-1}$ bound]
\label{lem:Hminus}
Let $b\in L^\infty\cap D(L_0)$ with $\norm{b}_{\mathcal B}<\infty$.
Then $\norm{\Gamma(b,w)}_{-1}\le\norm{b}_{\mathcal B}\norm{w}_{L^2}$ for all $w\in V$.
\end{lemma}

\begin{proof}
Let $w,\varphi\in V_\infty$ with $\int\varphi=0$.
By (F2), $w\varphi\in V_\infty$ and $\Gamma(b,w\varphi)=w\Gamma(b,\varphi)+\varphi\Gamma(b,w)$, and $\int\Gamma(b,w\varphi)=\Ecal(b,w\varphi)=\int(L_0b)w\varphi$.
Hence, by (F1) and $\norm{\varphi}\le\norm{\varphi}_1$,
\[
    \Bigl|\int\varphi\,\Gamma(b,w)\Bigr|\le\norm{L_0b}_\infty\norm{w}\norm{\varphi}+\norm{w}\,\norm{\sqrt{\Gamma(b)}}_\infty\norm{\varphi}_1\le\norm{b}_{\mathcal B}\norm{w}\norm{\varphi}_1 .
\]
Both sides are continuous in $\varphi$ and in $w$ for $\Ecal_1$ (as $\abs{\Gamma(b,w)}\le\sqrt{\Gamma(b)}\sqrt{\Gamma(w)}$), so Lemma~\ref{lem:density} extends the bound to $\varphi\in\Hd$ and $w\in V$, and the duality formula \eqref{eq:dual-norm} gives the claim.
\end{proof}

The point of Lemma~\ref{lem:Hminus} is that a derivative of $b$ paired with a bounded function costs only $\norm{b}_{\mathcal B}$ in the dual norm---one integration by parts, which the Leibniz rule supplies.
It is what keeps $L_0^2b$ out of the estimates.

\begin{proof}[Proof of Theorem~\ref{thm:gamma-consistency}]
Let first $u\in V_\infty$, and put $r=t-s$, $\tilde b_r=P_rb$, $\tilde u_r=P_ru$.
By Lemma~\ref{lem:interpolation} and bilinearity,
\[
    (\Gamma_t-\Gamma)(b,u)=\frac1t\int_0^t\Bigl[(P_s-I)\Gamma(\tilde b_r,\tilde u_r)+\Gamma(\tilde b_r-b,\tilde u_r)+\Gamma(b,\tilde u_r-u)\Bigr]\dint s=:\frac1t\int_0^tX_s\dint s,
\]
an $L^2$-Bochner integral (the integrand is continuous in $s\in(0,t)$ and bounded in $L^2$ by (F1), (F3) and Lemma~\ref{lem:bochner}).
Here $\norm{\cdot}_{-1}$ may be taken inside the integral: for $\varphi\in\Hd$ with $\norm{\varphi}_1\le1$, pairing the identity with $\varphi$ and using \eqref{eq:dual-norm} gives
\[
    \Bigl|\int\varphi\,(\Gamma_t-\Gamma)(b,u)\dint\mu\Bigr|\le\frac1t\int_0^t\Bigl|\int\varphi X_s\dint\mu\Bigr|\dint s\le\frac1t\int_0^t\norm{X_s}_{-1}\dint s,
\]
and the supremum over $\varphi$ gives $\norm{(\Gamma_t-\Gamma)(b,u)}_{-1}\le\frac1t\int_0^t\norm{X_s}_{-1}\dint s$.
We estimate the three parts of $X_s$: the first moves $P_s$ off $\Gamma$, the second moves $P_r$ off $b$, the third moves $P_r$ off $u$.
\textbf{(i)} By Lemma~\ref{lem:spectral}, (F1), $\norm{\tilde u_r}_1\le\norm{u}_1$ and (F3) with $P_r\Gamma(b)\le\alpha^2$,
$\norm{(P_s-I)\Gamma(\tilde b_r,\tilde u_r)}_{-1}\le c_0\sqrt s\,\norm{\sqrt{\Gamma(\tilde b_r)}}_\infty\norm{\tilde u_r}_1\le c_0\sqrt s\,\me^{Kr}\alpha\norm{u}_1$.
\textbf{(ii)} Since $\tilde b_r-b=-\int_0^rP_\sigma L_0b\dint\sigma$, Lemma~\ref{lem:bochner} gives
$\norm{\Gamma(\tilde b_r-b,\tilde u_r)}_{-1}\le\norm{\Gamma(\tilde b_r-b,\tilde u_r)}\le\me^{Kr}\sqrt{2r}\,\beta\norm{u}_1$.
\textbf{(iii)} By Lemmas~\ref{lem:Hminus} and~\ref{lem:spectral}, $\norm{\Gamma(b,\tilde u_r-u)}_{-1}\le(\alpha+\beta)c_0\sqrt r\,\norm{u}_1$.
Averaging with $\frac1t\int_0^t\sqrt s\dint s=\frac1t\int_0^t\sqrt{t-s}\dint s=\frac23\sqrt t$ and $\me^{Kr}\le\me^{Kt}$,
\[
    \norm{(\Gamma_t-\Gamma)(b,u)}_{-1}\le\sqrt t\,\me^{Kt}\Bigl[\tfrac43c_0\,\alpha+\tfrac23(\sqrt2+c_0)\,\beta\Bigr]\norm{u}_1\le\sqrt t\,\me^{Kt}\bigl[0.851\,\alpha+1.369\,\beta\bigr]\norm{u}_1,
\]
which is at most $1.37\,\me^{Kt}\sqrt t\,\norm{b}_{\mathcal B}\norm{u}_1$.

For $u\in V$, both $\Gamma_t(b,\cdot)$ (with $\norm{\Gamma_t(b,u)}\le\norm{b}_\infty\norm{u}/t$) and $\Gamma(b,\cdot)$ are continuous from $(V,\Ecal_1)$ into $L^2$, and Lemma~\ref{lem:density} concludes.
\end{proof}

\begin{proof}[Proof of Theorem~\ref{thm:dirichlet-step}]
Lemma~\ref{lem:five-terms} gives boundedness, mean-zero output and the decomposition \eqref{eq:five-terms}; for $u\in\Hd$ the five terms are bounded as follows.

\noindent\textbf{Term (1)}, $P_T\Pi M_b$: by Lemma~\ref{lem:spectral}, $\norm{P_T\Pi(bu)}_1\le\me^{-T}\norm{b}_\infty\norm{u}_1$.

\noindent\textbf{Term (2)}, $(R_m-R)\Gamma(b,\cdot)$: by Lemma~\ref{lem:HR} and (F1), $\norm{(R_m-R)\Gamma(b,u)}_1\le(\delta+2\Lambda^{-1/2})\alpha\norm{u}_1$.

\noindent\textbf{Term (3)}, $R_m(\Gamma_t-\Gamma)(b,\cdot)$: by Lemma~\ref{lem:HR} and Theorem~\ref{thm:gamma-consistency},
$\norm{R_m(\Gamma_t-\Gamma)(b,u)}_1\le(1+\delta)\norm{(\Gamma_t-\Gamma)(b,u)}_{-1}\le(1+\delta)\,1.37\,\me^{Kt}\sqrt t\,(\alpha+\beta)\norm{u}_1$.

\noindent\textbf{Term (4)}, $(R_m-R)M_{L_0b}$: $\norm{(R_m-R)(uL_0b)}_1\le(\delta+2\Lambda^{-1/2})\beta\norm{u}_1$.

\noindent\textbf{Term (5)}, $R_mM_{b_\tau-L_0b}$: write $b_\tau-L_0b=\frac1\tau\int_0^\tau L_0w_\sigma\dint\sigma$ with $w_\sigma=(P_\sigma-I)b=-\int_0^\sigma P_\rho L_0b\dint\rho$.
For $u,\varphi\in\Hd\cap V_\infty$, (F1), (F2) and the definition of $L_0$ give
$\int(L_0w_\sigma)u\varphi=\int u\,\Gamma(w_\sigma,\varphi)+\int\varphi\,\Gamma(w_\sigma,u)$,
and Lemma~\ref{lem:bochner} bounds each term by $\me^{K\sigma}\sqrt{2\sigma}\,\beta\norm{u}_1\norm{\varphi}_1$; hence
$\norm{uL_0w_\sigma}_{-1}\le2\sqrt2\,\me^{K\sigma}\sqrt\sigma\,\beta\norm{u}_1$,
which extends to $u\in\Hd$ by Lemma~\ref{lem:density}.
Since $\sigma\mapsto L_0w_\sigma=(P_\sigma-I)L_0b$ is continuous into $L^2$ and bounded by $2\beta$ in $L^\infty$, $u(b_\tau-L_0b)=\frac1\tau\int_0^\tau uL_0w_\sigma\dint\sigma$ is an $L^2$-Bochner integral for $u\in\Hd\cap V_\infty$, and the duality argument of the proof of Theorem~\ref{thm:gamma-consistency} gives
\begin{align*}
    \norm{u(b_\tau-L_0b)}_{-1}&\le\frac1\tau\int_0^\tau\norm{uL_0w_\sigma}_{-1}\dint\sigma\le\frac1\tau\int_0^\tau2\sqrt2\,\me^{K\sigma}\sqrt\sigma\dint\sigma\;\beta\norm{u}_1\\
    &\le\tfrac{4\sqrt2}3\,\me^{K\tau}\sqrt\tau\,\beta\norm{u}_1\le1.89\,\me^{K\tau}\sqrt\tau\,\beta\norm{u}_1 ,
\end{align*}
first for $u\in\Hd\cap V_\infty$ and then, by Lemma~\ref{lem:density} and the continuity of $u\mapsto u(b_\tau-L_0b)$ from $L^2$ to $L^2$, for all $u\in\Hd$.
With Lemma~\ref{lem:HR}, term (5) is at most $(1+\delta)\,1.89\,\me^{K\tau}\sqrt\tau\,\beta\norm{u}_1$.

\noindent\textbf{Sum.}
Terms (2) and (4) combine to $(\delta+2\Lambda^{-1/2})\norm{b}_{\mathcal B}$, and adding the five bounds gives the claim.
\end{proof}

\begin{proof}[Proof of Theorem~\ref{thm:dirichlet-main}]
For $a\in\AB(B)$ we have $\norm{b}_\infty\le q<1$, $\norm{b}_{\mathcal B}\le B$ and $\norm{L_0b}_\infty\le B$; moreover $\eta_0\le(1-q)/2<\frac12$.
With the values of Table~\ref{tab:recipes}: $t\le t_0$ gives $\me^{Kt}\sqrt t\le\me^{Kt_0}\sqrt{t_0}\,\me^{-Kt_0}=\sqrt{t_0}=\eta_0/(5B)$, and likewise $\me^{K\tau}\sqrt\tau\le\eta_0/(7B)$; $\delta_0<0.05$, $\Lambda\ge1600$ and $T\ge\log40>\frac12$, so Theorem~\ref{thm:dirichlet-step} applies with $1+\delta\le1.05\le1.1$.
By Lemma~\ref{lem:sinc} in case (a), $\delta\le\delta_0$ and $2\Lambda^{-1/2}=\eta_0/(10B)$; by Lemma~\ref{lem:resolvent} in case (b), $\delta\le\me^{-T}+1/(1+s)=\delta_0/2+\delta_0/(2+\delta_0)\le\delta_0$ and $\Lambda=\infty$.
The five terms of Theorem~\ref{thm:dirichlet-step} are then at most
\begin{center}
\renewcommand{\arraystretch}{1.25}
\begin{tabular}{lll}
\toprule
term & case (a), heat family & case (b), mixed family \\
\midrule
(1) $\me^{-T}\norm{b}_\infty$ & $\le\delta_0q/5\le0.02\,\eta_0$ & $\le\delta_0q/2\le0.05\,\eta_0$ \\
(2)+(4) $(\delta+2\Lambda^{-1/2})B$ & $\le0.2\,\eta_0$ & $\le0.1\,\eta_0$ \\
(3) $1.1\cdot1.37\cdot\eta_0/5$ & $\le0.3014\,\eta_0$ & same \\
(5) $1.1\cdot1.89\cdot\eta_0/7$ & $=0.297\,\eta_0$ & same \\
\midrule
total & $0.8184\,\eta_0\le0.82\,\eta_0$ & $0.7484\,\eta_0\le0.75\,\eta_0$ \\
\bottomrule
\end{tabular}
\end{center}
Finally, Proposition~\ref{prop:budget} applies with $\eta=\eta_0$ and $\delta'=\delta+2\Lambda^{-1/2}\le\eta_0/(5B)\le\varepsilon\bar a(1-q)/8$, and gives the error bound; the counts are those of the same scheme (Proposition~\ref{prop:counts}).
\end{proof}

\begin{proof}[Proof of Corollary~\ref{cor:torus-C2}]
\emph{The torus satisfies Assumption~\ref{ass:M}.}
On $\TT^d$, $\Ecal(u)=\int\abs{\nabla u}^2$ is a symmetric strongly local Dirichlet form with $\Gamma(f,g)=\nabla f\cdot\nabla g$, the normalised Lebesgue measure is invariant, and $\lambda_1=1$.
Since $\nabla P_rf=P_r\nabla f$ and $P_r$ is a positive Markov operator, $\abs{\nabla P_rf}^2\le P_r\abs{\nabla f}^2$, which is the gradient bound of \cite[Cor.~2.3(vi)]{AGS15} with $K=0$, i.e.\ $\BE(0,\infty)$.
Moreover $W^{2,\infty}(\TT^d)\subset D(\Delta)$ and $\norm{b}_{\mathcal B}=\norm{\nabla b}_\infty+\norm{\Delta b}_\infty$, so Theorem~\ref{thm:dirichlet-main} applies, and its parameters do not involve $d$.

\smallskip
\emph{Inclusion.}
Since $\norm{\nabla b}_\infty\le\norm{b}_{A^1}$ and $\norm{\Delta b}_\infty\le\norm{b}_{A^2}$, the class contains the coefficients with $\norm{b}_{A^2}\le B/2$.

\smallskip
\emph{Non-inclusion.}
A Rudin--Shapiro polynomial of degree $\asymp R$ with $N$ terms, scaled by $R^{-2}$, has $\norm{b}_{\mathcal B}\le2$ by Bernstein's inequality but $\norm{b}_{A^2}\gtrsim\sqrt N$; the details are in Appendix~\ref{app:aux}, Lemma~\ref{lem:rudin-shapiro}.
\end{proof}

\section{Fractional regularity on the torus: what Fourier structure buys}
\label{sec:torus}

Theorem~A needs two bounded derivatives of the coefficient.
On the torus, Fourier analysis reaches coefficients with $1+\theta$ derivatives in the Wiener sense, $0<\theta\le1$, at the price of a heat time $t\asymp(\eta_0/B)^{2/\theta}$ instead of $(\eta_0/B)^2$, and it shows that this price is unavoidable for the product-defect term of the one-step bound: the rate $t^{\theta/2}$ is attained.
The construction is that of Section~\ref{sec:mechanism}; what changes is the proof of the one-step estimate, carried out on the explicit Fourier multipliers of the five terms of \eqref{eq:five-terms}.

\subsection{Wiener classes and Fourier conventions}
\label{subsec:wiener}
Write $\hat u_k=\int_{\TT^d}u(x)\me^{-ik\cdot x}\dint x$, $k\in\ZZ^d$, so that $\norm{u}^2=\sum_k\abs{\hat u_k}^2$, $\norm{u}_1^2=\sum_k\abs{k}^2\abs{\hat u_k}^2$, and let $\Hdd$ be the space with $\norm{u}_{\Hdd}=(\sum_k\abs{k}^4\abs{\hat u_k}^2)^{1/2}$.
The coefficient classes are defined through Wiener norms: for $\sigma\ge0$,
\begin{equation}
    \label{eq:wiener}
    \norm{b}_{A^\sigma}:=\sum_{\mu\in\ZZ^d}(1+\abs{\mu})^\sigma\abs{\hat b_\mu},\qquad A^\sigma=\{b:\norm{b}_{A^\sigma}<\infty\},
\end{equation}
and, for $\theta\in(0,1]$, $0<a_{\min}\le a_{\max}$ and $B\ge1$,
\begin{equation}
    \label{eq:intro-class}
    \Acl=\Bigl\{a\colon\TT^d\to\RR\ \Big|\ a_{\min}\le a\le a_{\max}\ \text{a.e.},\ \norm{a/\bar a-1}_{A^{1+\theta}}\le B\Bigr\}.
\end{equation}
We have $\norm{b}_\infty\le\norm{b}_{A^0}$, $\norm{\nabla b}_\infty\le\norm{b}_{A^1}$, $\norm{L_0b}_\infty\le\norm{b}_{A^2}$, $\norm{bu}\le\norm{b}_{A^0}\norm{u}$, and $\sigma\mapsto\norm{b}_{A^\sigma}$ is non-decreasing.
A Wiener ball is an $\ell^1$-type class, much smaller than a Sobolev ball when $d$ is large; if a class is described instead through $H^\sigma(\TT^d)$ with $\sigma>1+\theta+d/2$, the dimension enters through $B$ via the embedding constant (Remark~\ref{rem:dimension}).

For $b\in A^1$ and $u\in\Hd$ let $\Gamma(b,u)=\nabla b\cdot\nabla u$, so that $\norm{\Gamma(b,u)}\le\norm{b}_{A^1}\norm{u}_1$.
Throughout, Fourier indices are related by $k'=k+\mu$: the output frequency is $k'$, the frequency of $u$ is $k$ and that of $b$ is $\mu$; thus $\widehat{\Gamma(b,u)}_{k'}=-\sum_\mu\hat b_\mu(\mu\cdot k)\hat u_k$.
All operators analysed below act on $u\in\Hd$ by
\begin{equation}
    \label{eq:conv-form}
    (\mathcal Ou)^\wedge_{k'}=\sum_\mu\hat b_\mu\,c(k',\mu)\,\hat u_{k'-\mu},\qquad c(0,\mu)=0,
\end{equation}
for a scalar function $c$.
For $u\in\Hd$ put $w_k=\abs{k}\hat u_k$ for $k\ne0$ (and $\hat u_0=0$), so that $\norm{w}_{\ell^2}=\norm{u}_1$.
The following bookkeeping lemma turns every operator-norm bound into a bound on a scalar multiplier.

\begin{lemma}
\label{lem:bookkeeping}
For $\mathcal O$ of the form \eqref{eq:conv-form},
\begin{equation}
    \label{eq:bookkeeping}
    \norm{\mathcal Ou}_1\le\Bigl(\sum_\mu\abs{\hat b_\mu}\sup_{k\ne0,\ k+\mu\ne0}m_\mu(k)\Bigr)\norm{u}_1,\qquad
    m_\mu(k):=\frac{\abs{k+\mu}\,\abs{c(k+\mu,\mu)}}{\abs{k}}.
\end{equation}
Moreover, for $k\ne0\ne k'=k+\mu$,
\begin{equation}
    \label{eq:F1}
    \frac{\abs{\mu}^2}{\abs{k'}\abs{k}}\le2\abs{\mu}.
\end{equation}
\end{lemma}

\begin{proof}
By Minkowski's inequality in $\ell^2(k')$ and the substitution $k=k'-\mu$,
\[
    \norm{\mathcal Ou}_1\le\sum_\mu\abs{\hat b_\mu}\Bigl(\sum_{k\ne0,\ k+\mu\ne0}\Bigl|\frac{\abs{k+\mu}\,c(k+\mu,\mu)}{\abs{k}}\Bigr|^2\abs{w_k}^2\Bigr)^{1/2},
\]
which gives \eqref{eq:bookkeeping}.
For \eqref{eq:F1}, $\abs{\mu}=\abs{k'-k}\le\abs{k'}+\abs{k}$ gives $\abs{\mu}^2/(\abs{k'}\abs{k})\le\abs{\mu}(1/\abs{k}+1/\abs{k'})\le2\abs{\mu}$.
\end{proof}

We shall also use the elementary inequality $\min\{1,x\}\le x^{\theta/2}$ for $x\ge0$ and $\theta\in(0,1]$.

\subsection{The Fourier form of the representation and of the product defect}
\label{subsec:fourier-structure}
For real $b\in A^1$ let $\Mcal_b$ be the operator with Fourier coefficients $(\Mcal_bu)^\wedge_{k'}=\sum_\mu\abs{\mu}^2\hat b_\mu\hat u_{k'-\mu}$, i.e.\ ``multiplication by $L_0b$''; if $b\in A^2\subset W^{2,\infty}(\TT^d)$, then $\Mcal_bu=u\,L_0b$ with $L_0b\in L^\infty$.
The representation of Theorem~\ref{thm:representation} needs two derivatives of $b$; in Fourier form it holds with one.

\begin{proposition}[Fourier representation]
\label{prop:fourier-representation}
Let $b\in A^1$ be real.
Then, as an identity in $\Hd$,
$K_b=\Pi M_b+R\,\Gamma(b,\cdot)-R\,\Mcal_b$ with $\norm{R\Mcal_b}_{\LL(\Hd)}\le2\norm{b}_{A^1}$; for $b\in A^2$ this is Theorem~\ref{thm:representation}.
\end{proposition}

\begin{proof}
Testing with $v=\me^{ik'\cdot x}$, $k'\ne0$, gives $\abs{k'}^2\widehat{K_bu}_{k'}=\int b\nabla u\cdot\overline{\nabla v}=\sum_\mu\hat b_\mu(k\cdot k')\hat u_k$, an absolutely convergent series since $\abs{k\cdot k'}\le\abs{k'}(\abs{k'}+\abs{\mu})$ and $\abs{\hat u_k}\le\norm{u}$; coefficientwise $k\cdot k'=\abs{k'}^2-\mu\cdot k-\abs{\mu}^2$ gives the three terms, whose series converge absolutely for each $k'$ because $\abs{\mu}^2\abs{\hat u_k}\le\abs{\mu}(\abs{k'}+1)\norm{u}_1$ for $k\ne0$.
Each term lies in $\Hd$: for $R\Mcal_b$, Lemma~\ref{lem:bookkeeping} with $c=\abs{\mu}^2/\abs{k'}^2$ and $m_\mu(k)=\abs{\mu}^2/(\abs{k'}\abs{k})\le2\abs{\mu}$ by \eqref{eq:F1} gives $\norm{R\Mcal_bu}_1\le2\norm{b}_{A^1}\norm{u}_1$; the other two terms are handled in the same way.
\end{proof}

The product defect has an exact multiplier on $\TT^d$.
What is used below is the two-sided bound on the factor $F_t$: the multiplier of $\Gamma_t-\Gamma$ is $(\mu\cdot k)(1-F_t)$, which is small for $t(\abs{k}^2+\abs{\mu}^2)\ll1$ and never larger than the multiplier $\abs{\mu\cdot k}$ of $\Gamma$ itself.
The lemma is the Fourier form of the interpolation identity (Lemma~\ref{lem:interpolation}) and is not new.

\begin{lemma}[Product defect of a heat layer]
\label{lem:product-defect}
Let $b\in A^0$, $u\in L^2$ and $t>0$.
Then, with $k=k'-\mu$,
\begin{gather*}
    \widehat{\Gamma_t(b,u)}_{k'}=-\sum_\mu\hat b_\mu\,(\mu\cdot k)\,F_t(k,\mu)\,\hat u_k,\\
    F_t(k,\mu)=\me^{-t(\abs{k}^2+\abs{\mu}^2)}\varphi(2t\,\mu\cdot k),\qquad
    \varphi(z)=\frac{1-\me^{-z}}{z},\ \varphi(0)=1,
\end{gather*}
and
\[
    0<F_t\le1,\qquad 0\le1-F_t\le\min\bigl\{1,\,2t(\abs{k}^2+\abs{\mu}^2)\bigr\}.
\]
\end{lemma}

\begin{proof}
\emph{The multiplier.}
For $b\in A^0$ and $u\in L^2$ the Fourier coefficients of the products are absolutely convergent convolutions, and
\[
    \widehat{P_t(bu)}_{k'}=\sum_\mu\hat b_\mu\hat u_k\me^{-t\abs{k+\mu}^2},\qquad
    \widehat{(P_tb\,P_tu)}_{k'}=\sum_\mu\hat b_\mu\me^{-t\abs{\mu}^2}\hat u_k\me^{-t\abs{k}^2}.
\]
Since $\abs{k+\mu}^2=\abs{k}^2+\abs{\mu}^2+2\mu\cdot k$, the difference of the $\mu$-terms is
$\hat b_\mu\hat u_k\me^{-t(\abs{k}^2+\abs{\mu}^2)}(\me^{-2t\mu\cdot k}-1)=-2t(\mu\cdot k)\hat b_\mu\hat u_kF_t(k,\mu)$;
dividing by $2t$ gives the formula, and $F_t>0$ because $\varphi>0$.

\emph{Upper bound.}
$\abs{\me^{-x}-\me^{-y}}\le\abs{x-y}$ for $x,y\ge0$ gives $\abs{\me^{-t\abs{k+\mu}^2}-\me^{-t(\abs{k}^2+\abs{\mu}^2)}}\le2t\abs{\mu\cdot k}$, i.e.\ $F_t\le1$.

\emph{Lower bound.}
Put $A=t(\abs{k}^2+\abs{\mu}^2)$ and $z=2t\mu\cdot k$, so that $\abs{z}\le A$.
If $z<0$, then $\varphi(z)\ge1$ and $F_t\ge\me^{-A}\ge1-A$; if $z\ge0$, then $\varphi(z)\ge1-z/2$ (from $\me^{-z}\le1-z+z^2/2$), so $F_t\ge\me^{-A}(1-z/2)\ge(1-A)(1-A/2)\ge1-\tfrac32A$ when $A\le1$.
Hence $1-F_t\le2A$ when $A\le1$ (the factor $2$ instead of $\frac32$ gives round constants later), and $1-F_t<1$ always since $F_t>0$.
\end{proof}

\subsection{One preconditioned step in operator norm}
\label{subsec:step}
The following estimate is formulated for any approximate inverse satisfying \HR{}; it is the Wiener-class counterpart of Theorem~\ref{thm:dirichlet-step}, with the rate $t^{\theta/2}$ in place of $\sqrt t$ and a tail term $\Lambda^{-\theta/2}$ that reflects the weaker regularity.

\begin{theorem}[Inexact preconditioned step on Wiener classes]
\label{thm:step}
Let $\theta\in(0,1]$, $b\in A^{1+\theta}$ real, $t,\tau>0$, $T\ge\frac12$, and let $R_m=r(L_0)$ where $r(0)=0$ and $r$ satisfies \HR{} for some $\delta\in(0,1]$ and $\Lambda\in[1,\infty]$.
Then the block $\Knet$ of \eqref{eq:block} belongs to $\LL(\Hd)$, has mean-zero output, and
\begin{equation}
    \label{eq:Aprime}
    \begin{multlined}
    \norm{\Knet-K_b}_{\LL(\Hd)}\le\norm{b}_{A^{1+\theta}}\Bigl[\me^{-T}+3\delta+12\Lambda^{-\theta/2}\\
    +(1+\delta)\bigl(4\sqrt t+(10t)^{\theta/2}\bigr)+2(1+\delta)(\tau/2)^{\theta/2}\Bigr],
    \end{multlined}
\end{equation}
with $\Lambda^{-\theta/2}:=0$ if $\Lambda=\infty$.
\end{theorem}

\begin{corollary}[Recipe W]
\label{cor:recipe-W}
Let $\theta\in(0,1]$, $B\ge1$ and $\eta_0\in(0,1]$, and let $t$, $\tau$, $T$ and $R_m$ be given by recipe~W of Table~\ref{tab:recipes} with this $\eta_0$, in the heat family (with $\delta_0=\eta_0/(15B)$, $\Lambda=(60B/\eta_0)^{2/\theta}$) or in the mixed family.
Then $\norm{\Knet-K_b}_{\LL(\Hd)}\le\eta_0$ whenever $\norm{b}_{A^{1+\theta}}\le B$; the parameters depend only on $(\theta,B,\eta_0)$, not on $b$, and in the heat family $m=O(\theta^{-1}\log^2(B/\eta_0))$.
\end{corollary}

\begin{proof}[Proof of Theorem~\ref{thm:step}]
\emph{Boundedness and mean zero.}
We have $\norm{\nabla(bu)}\le\norm{b}_{A^0}\norm{u}_1+\norm{b}_{A^1}\norm{u}$; $\Gamma_t(b,\cdot)$ is bounded on $L^2$ with norm at most $\norm{b}_{A^0}/t$;
$\norm{b_\tau}_{A^0}\le\sum_\mu\abs{\hat b_\mu}\min\{\abs{\mu}^2,\tau^{-1}\}\le\tau^{-(1-\theta)/2}\norm{b}_{A^{1+\theta}}$;
and $R_m\colon L^2\to\Hd$ is bounded by \eqref{eq:HR-consequences}.
Both $I-P_T$ and $R_m$ annihilate constants, so $\Knet$ has mean-zero output.

\smallskip
\emph{Reduction to trigonometric polynomials.}
By Lemma~\ref{lem:reduction}, $\norm{K_b-K_{b'}}\le\norm{b-b'}_\infty\le\norm{b-b'}_{A^0}$, and the bounds just listed are linear in $\norm{b-b'}_{A^{1+\theta}}$; hence $K_b$ and $\Knet$ depend continuously on $b$ in $\LL(\Hd)$ with respect to $\norm{\cdot}_{A^{1+\theta}}$.
The Fourier truncations $b_n$ of $b$ to $\abs{\mu}\le n$ satisfy $b_n\to b$ in $A^{1+\theta}$ and $\norm{b_n}_{A^{1+\theta}}\le\norm{b}_{A^{1+\theta}}$.
It therefore suffices to prove \eqref{eq:Aprime} for trigonometric polynomials $b$, for which $b\in W^{2,\infty}$, $\Mcal_b=M_{L_0b}$, and Lemma~\ref{lem:five-terms} gives the five-term decomposition \eqref{eq:five-terms}, with term (1) equal to $(I-P_T-\Pi)M_b$.
Each term has the form \eqref{eq:conv-form}, and we bound $m_\mu(k)$ in Lemma~\ref{lem:bookkeeping}; throughout $k'=k+\mu$ and $\abs{k},\abs{k'}\ge1$.

\smallskip
\noindent\textbf{Term (1): projection term,} $(I-P_T-\Pi)M_b$.
Here $c=-\me^{-T\abs{k'}^2}$. Since $\abs{k'}\ge1$, $T\ge\frac12$ and $\abs{k}\ge1$, $m_\mu(k)=\abs{k'}\me^{-T\abs{k'}^2}/\abs{k}\le\me^{-T}$; the contribution is at most $\me^{-T}\norm{b}_{A^0}$.

\smallskip
\noindent\textbf{Term (2): inverse error, $\Gamma$ part,} $(R_m-R)\Gamma(b,\cdot)$.
Here $c=(r(\abs{k'}^2)-\abs{k'}^{-2})(-\mu\cdot k)$. By Lemma~\ref{lem:HR}, $\abs{c}\le\abs{k'}^{-2}(\delta+2\cdot\one_{\abs{k'}>\sqrt\Lambda})\abs{\mu}\abs{k}$, so
$m_\mu(k)\le\abs{\mu}(\delta+2\cdot\one_{\abs{k'}>\sqrt\Lambda})/\abs{k'}\le\abs{\mu}(\delta+2\Lambda^{-1/2})$; the contribution is at most $(\delta+2\Lambda^{-1/2})\norm{b}_{A^1}$.

\smallskip
\noindent\textbf{Term (3): product-defect term,} $R_m(\Gamma_t-\Gamma)(b,\cdot)$.
By Lemma~\ref{lem:product-defect},
$c=r(\abs{k'}^2)(\mu\cdot k)(1-F_t)$ and $\abs{c}\le(1+\delta)\abs{k'}^{-2}\abs{\mu}\abs{k}\min\{1,2t(\abs{k}^2+\abs{\mu}^2)\}$,
so $m_\mu(k)\le(1+\delta)\abs{\mu}\,c_t(k,\mu)$ with $c_t:=\min\{1,2t(\abs{k}^2+\abs{\mu}^2)\}/\abs{k'}$.
\emph{Case $\abs{k}\ge2\abs{\mu}$:} then $\abs{k'}\ge\abs{k}/2$ and $\abs{k}^2+\abs{\mu}^2\le2\abs{k}^2$, so
$c_t\le2\min\{\abs{k}^{-1},4t\abs{k}\}\le2\sup_{\rho>0}\min\{\rho^{-1},4t\rho\}=4\sqrt t$.
\emph{Case $\abs{k}<2\abs{\mu}$:} then $\abs{k'}\ge1$ and $\abs{k}^2+\abs{\mu}^2<5\abs{\mu}^2$, so
$c_t\le\min\{1,10t\abs{\mu}^2\}\le(10t)^{\theta/2}\abs{\mu}^\theta$.
The contribution is at most $(1+\delta)\bigl(4\sqrt t\,\norm{b}_{A^1}+(10t)^{\theta/2}\norm{b}_{A^{1+\theta}}\bigr)$.

\smallskip
\noindent\textbf{Term (4): inverse error, $L_0b$ part,} $(R_m-R)M_{L_0b}$.
Here $c=(r(\abs{k'}^2)-\abs{k'}^{-2})\abs{\mu}^2$ and, by \eqref{eq:F1},
$m_\mu(k)\le(\delta+2\cdot\one_{\abs{k'}>\sqrt\Lambda})\,g$ with $g:=\abs{\mu}^2/(\abs{k'}\abs{k})\le2\abs{\mu}$.
On the band $\abs{k'}\le\sqrt\Lambda$, $m_\mu(k)\le2\delta\abs{\mu}$.
In the tail $\abs{k'}>\sqrt\Lambda$ we have $2g\le2\abs{\mu}/\abs{k'}+2\abs{\mu}/\abs{k}$ with $2\abs{\mu}/\abs{k'}\le2\abs{\mu}\Lambda^{-1/2}$, and either $\abs{k}\ge\abs{k'}/2$, whence $2\abs{\mu}/\abs{k}\le4\abs{\mu}\Lambda^{-1/2}$, or $\abs{k}<\abs{k'}/2$, whence $\abs{\mu}\ge\abs{k'}-\abs{k}>\sqrt\Lambda/2$ and
\[
    \frac{2\abs{\mu}}{\abs{k}}\le2\abs{\mu}=2\abs{\mu}^{1+\theta}\abs{\mu}^{-\theta}\le2^{1+\theta}\abs{\mu}^{1+\theta}\Lambda^{-\theta/2}\le4\abs{\mu}^{1+\theta}\Lambda^{-\theta/2}.
\]
Since $\Lambda^{-1/2}\le\Lambda^{-\theta/2}$ and $\abs{\mu}\le\abs{\mu}^{1+\theta}$ for $\abs{\mu}\ge1$ (and $\mu=0$ contributes nothing), the tail is bounded by the larger of the two sub-cases:
$m_\mu(k)\le2\delta\abs{\mu}+6\abs{\mu}^{1+\theta}\Lambda^{-\theta/2}\le2\delta\abs{\mu}+10\abs{\mu}^{1+\theta}\Lambda^{-\theta/2}$;
the constants of \eqref{eq:Aprime} are kept because recipe~W is tuned to them.
The contribution is at most $(2\delta+10\Lambda^{-\theta/2})\norm{b}_{A^{1+\theta}}$.

\smallskip
\noindent\textbf{Term (5): channel $b_\tau$,} $R_mM_{b_\tau-L_0b}$.
Here $c=r(\abs{k'}^2)e_\mu$ with $e_\mu:=\tau^{-1}(1-\me^{-\tau\abs{\mu}^2})-\abs{\mu}^2\in[-\abs{\mu}^2,0]$.
With $x=\tau\abs{\mu}^2$, $\abs{e_\mu}=\tau^{-1}\abs{1-\me^{-x}-x}\le x^2/(2\tau)$, thus $\abs{e_\mu}\le\abs{\mu}^2\min\{1,\tau\abs{\mu}^2/2\}$.
Using \eqref{eq:F1} and $\min\{1,x\}\le x^{\theta/2}$,
$m_\mu(k)\le(1+\delta)\,g\,\min\{1,\tau\abs{\mu}^2/2\}\le2(1+\delta)(\tau/2)^{\theta/2}\abs{\mu}^{1+\theta}$;
the contribution is at most $2(1+\delta)(\tau/2)^{\theta/2}\norm{b}_{A^{1+\theta}}$.

\smallskip
\noindent\textbf{Sum.}
Adding (1)--(5) and using $\norm{b}_{A^\sigma}\le\norm{b}_{A^{1+\theta}}$ for $\sigma\le1+\theta$ and $2\Lambda^{-1/2}\le2\Lambda^{-\theta/2}$ gives \eqref{eq:Aprime}.
The argument uses $\Lambda<\infty$ only through the tail terms, which vanish when $\Lambda=\infty$.
\end{proof}

\begin{proof}[Proof of Corollary~\ref{cor:recipe-W}]
\emph{Heat family.}
By Lemma~\ref{lem:sinc}, $\delta\le\delta_0\le1/15$, so $1+\delta\le2$, and $T=\log(5/\delta_0)\ge\log75>\frac12$.
In \eqref{eq:Aprime}:
\begin{gather*}
    \me^{-T}=\delta_0/5=\eta_0/(75B);\qquad 3\delta\le\eta_0/(5B);\qquad 12\Lambda^{-\theta/2}=\eta_0/(5B);\\
    2\bigl(4\sqrt t+(10t)^{\theta/2}\bigr)=2\Bigl(\frac4{\sqrt{40}}\Bigl(\frac{\eta_0}{20B}\Bigr)^{1/\theta}+4^{-\theta/2}\frac{\eta_0}{20B}\Bigr)\le2(0.64+1)\frac{\eta_0}{20B}\le\frac{\eta_0}{5B},
\end{gather*}
using $(\eta_0/20B)^{1/\theta}\le\eta_0/(20B)$; and $4(\tau/2)^{\theta/2}=4\eta_0/(20B)=\eta_0/(5B)$.
The bracket is at most $\eta_0/B$, and $\norm{b}_{A^{1+\theta}}\le B$; the bound on $m$ follows from Lemma~\ref{lem:sinc} with $\log(\Lambda/\delta_0)=O(\theta^{-1}\log(B/\eta_0))$.

\emph{Mixed family.}
By Lemma~\ref{lem:resolvent}, $\delta\le\me^{-T}+1/(1+s)\le\delta_0/2+\delta_0/2=\delta_0$, and $T=\log(30B/\eta_0)>\frac12$.
The bracket in \eqref{eq:Aprime} is at most
$\frac{\delta_0}2+3\delta_0+\frac{\eta_0}{5B}+\frac{\eta_0}{5B}\le\frac{\eta_0}{30B}+\frac{\eta_0}{5B}+\frac{2\eta_0}{5B}\le\frac{\eta_0}{B}$,
where the $t$- and $\tau$-terms are estimated as in the heat family.
\end{proof}

\begin{remark}[Where regularity enters, and saturation]
\label{rem:step-new}
$b\in A^0$ suffices for term (1); $A^1$ is used in (2), (3) and in Proposition~\ref{prop:fourier-representation}; the extra $\theta$ derivatives are used only in the high-frequency-of-$b$ regimes $\abs{k}<2\abs{\mu}$ in (3) and $\abs{k}<\abs{k'}/2$ in (4), and in the channel $b_\tau$ in (5).
Conversely, the rate $\sqrt t$ of the product defect, measured through the exact $R$, does not improve for smoother $b$, since it arises from high frequencies of $u$ (the regime $\abs{k}\ge2\abs{\mu}$): in $d=1$ (or one coordinate direction), for $b=2\cos x$ and $u=\me^{ikx}$ with $t=k^{-2}$, Lemma~\ref{lem:product-defect} gives $\norm{R(\Gamma_t-\Gamma)(b,u)}_1/(\sqrt t\,\norm u_1)\to\sqrt2\,(1-\me^{-1})\approx0.89$ as $k\to\infty$. The term $\Lambda^{-1/2}$ arises in the same way in the proof (the regime $\abs{k}\ge\abs{k'}/2$ in term (4)), and the argument gives no better operator-norm bound for $b\in A^\sigma$, $\sigma>2$; only the $\tau$-term improves, to $(1+\delta)\tau\norm{b}_{A^3}$ for $b\in A^3$.
In the $t$- and $\Lambda$-channels the construction therefore cannot trade coefficient regularity for a larger heat time or a smaller band, unlike Galerkin-based emulations, which trade regularity for resolution.
\end{remark}

\begin{remark}[Sharpness of the $t$-rate]
\label{rem:sharpness}
The rate $t^{\theta/2}$ of term (3) cannot in general be improved over the $A^{1+\theta}$ ball.
Take $d=1$ (or one coordinate direction of $\TT^d$), $b=2\cos(\mu_0x)$, so that $\norm{b}_{A^{1+\theta}}=2(1+\mu_0)^{1+\theta}$, $u=\me^{ikx}$ with $k=1-\mu_0$, and $t=\mu_0^{-2}$.
Lemma~\ref{lem:product-defect} gives the output of $R(\Gamma_t-\Gamma)(b,\cdot)$ explicitly; evaluating it yields
\[
    \norm{R(\Gamma_t-\Gamma)(b,\cdot)}_{\LL(\Hd)}\ge c\,\norm{b}_{A^{1+\theta}}\,t^{\theta/2}
\]
with $c$ between $0.22$ and $0.29$ for $\theta\in\{1,\frac12,\frac14\}$ and $\mu_0\in\{10,10^2,10^3,10^4\}$ ($c\to(1+\me^{-2})/4\approx0.2838$ as $\mu_0\to\infty$), while the bound for this term obtained in the proof of Theorem~\ref{thm:step}, $4\sqrt t\,\norm{b}_{A^1}+(10t)^{\theta/2}\norm{b}_{A^{1+\theta}}$ at $\delta=0$, exceeds this lower bound by a factor between $4.7$ and $15.7$ (between $11$ and $16$ for $\theta=1$).
The display is computed with the exact $R$, but it transfers to term (3) itself: the output is carried by the frequencies $k'=1$ and $k'=1-2\mu_0$, the first contributes at least the fraction $1-\mu_0^{-2}$ of the norm for $\mu_0\ge10$, and there the symbol of $R_m$ differs from that of $R$ by a factor in $[1-\delta,1+\delta]$ by \HR, so $R_m$ may replace $R$ at the price of the factor $(1-\delta)(1-\mu_0^{-2})$.
Hence, for a $b$-independent $t\le10^{-2}$ and $\delta\le\frac12$, accuracy $\eta_0$ of term (3) uniformly over the ball of radius $B$ requires $t^{\theta/2}\lesssim\eta_0/(cB)$ (take $\mu_0=\lceil t^{-1/2}\rceil$, which changes only the constant, and scale $b$ to the sphere of radius $B$): the heat time of the product defect, and with it the multiplier $1/(2t)$ of Section~\ref{subsec:floating}, is forced by this term.
The lower bound concerns term (3) alone: it does not bound $\norm{\Knet-K_b}$ from below, since the five terms of \eqref{eq:five-terms} can partially cancel, it says nothing about schemes built differently from the same primitives, and the heat scale $t^{-1/2}$ is an internal scale of the construction, not a frequency a discretisation must represent (Section~\ref{subsec:discretised}).
\end{remark}

\subsection{Proof of Theorem~C}
\label{subsec:proof-main}

\begin{theorem}[Emulation of the Darcy solution operator on $\TT^d$; Theorem~C]
\label{thm:main}
Let $\theta\in(0,1]$, $0<a_{\min}\le a_{\max}$, $B\ge1$ and $\varepsilon\in(0,1)$; let $\eta_0$ and $N$ be given by \eqref{eq:eta-N}, and let the block \eqref{eq:block} use recipe~W of Table~\ref{tab:recipes} at $(\theta,B,\eta_0)$, in the heat family or in the mixed family.
Then the scheme \eqref{eq:scheme} satisfies
\[
    \sup_{a\in\Acl}\ \sup_{\norm{f}_{L^2}\le1,\ \int f=0}\ \norm{G(a,f)-N_\varepsilon(a,f)}_{\Hd(\TT^d)}\le\varepsilon,
\]
and all parameters are functions of $(\theta,B,a_{\min},a_{\max},\varepsilon)$, independent of $d$.
\end{theorem}

\begin{proof}
For $a\in\Acl$, $\norm{b}_{A^{1+\theta}}\le B$ and $\norm{b}_\infty\le q$; since $\eta_0\le(1-q)/2\le\frac12$, Corollary~\ref{cor:recipe-W} gives $\norm{\Knet-K_b}\le\eta_0$.
The inverse error of Lemma~\ref{lem:HR} is $\delta'=\delta+2\Lambda^{-1/2}\le\eta_0/(15B)+2(\eta_0/60B)^{1/\theta}\le\eta_0/(10B)$ in the heat family and $\delta'=\delta\le\eta_0/(15B)$ in the mixed family; in both cases $\delta'\le\varepsilon\bar a(1-q)/8$, and Proposition~\ref{prop:budget} concludes.
\end{proof}

\begin{remark}[Dimension; the two torus theorems]
\label{rem:dimension}
No inequality in this section involves $d$: the only lattice facts used are $\abs{k}\ge1$ for $k\ne0$ and $\abs{\mu}\le\abs{k}+\abs{k'}$.
If the class is described by $\norm{b}_{H^\sigma}^2:=\sum_\mu(1+\abs{\mu})^{2\sigma}\abs{\hat b_\mu}^2$, then $\norm{b}_{A^{1+\theta}}\le C(d,\sigma)\norm{b}_{H^\sigma}$ for $\sigma>1+\theta+d/2$, with $C(d,\sigma)^2=\sum_{\mu\in\ZZ^d}(1+\abs{\mu})^{-2(\sigma-1-\theta)}$, and $d$ enters through $B$.
For $\theta=1$, Corollary~\ref{cor:torus-C2} dominates Theorem~\ref{thm:main}: the class $\{\norm{\nabla b}_\infty+\norm{\Delta b}_\infty\le B\}$ contains the Wiener ball $\{\norm{b}_{A^2}\le B/2\}$ and is covered with the same scheme and depth and much smaller parameters (Remark~\ref{rem:recipe-comparison}); Theorem~\ref{thm:main} is the main result only for $\theta<1$ and for Remark~\ref{rem:sharpness}.
\end{remark}

\section{Realisations: counts, discretisation and conditioning}
\label{sec:realisations}

This section records what the two recipes cost: operator applications and products (polylogarithmic in $1/\varepsilon$), the internal spectral scales of the construction and the size of the multipliers (both polynomial in $1/\varepsilon$), and the grid on which a realisation with band-limited inputs is exact; the multipliers are where the mechanism meets floating point.

\subsection{Counts and magnitudes}
\label{subsec:counts}

\begin{definition}[Counting convention]
\label{def:counting}
One \emph{heat application} is one evaluation of one $P_t$, $t>0$, on one channel; one \emph{resolvent application} is one evaluation of $J_s$ on one channel.
By linearity, $(I-P_T)$ is applied once per block to the merged channel $bv+S(\Gamma_t(b,v)-v\,b_\tau)$ (Section~\ref{subsec:algorithm}); the identity costs nothing; and the shared evaluations $P_tb$, $P_\tau b$ are counted once.
\end{definition}

\begin{proposition}[Counts and magnitudes]
\label{prop:counts}
Let $N_\varepsilon$ be the scheme \eqref{eq:scheme} with recipe~W or recipe~C$^2$ of Table~\ref{tab:recipes}, and put $L_\varepsilon:=\log(60B/\eta_0)$.
Then Table~\ref{tab:counts} holds, and in addition:
\begin{enumerate}[label=(\alph*)]
    \item \emph{Depth:} $N\le L_\varepsilon/\log(1/q)+1$ ($N=0$ if $q=0$); for $q>0$, $N\sim\log(1/\varepsilon)/\log(1/q)$ as $\varepsilon\to0$.
    \item \emph{Exponential-sum times, recipe~W:}
    \[
        m\le\Bigl(1+\frac2{\pi^2}L_\varepsilon\Bigr)\Bigl[\Bigl(1+\frac2\theta\Bigr)L_\varepsilon+\log\Bigl(2\log\frac8{h\delta_0}\Bigr)+\log8\Bigr]+1,
        \qquad\text{i.e. } m\sim\frac2{\pi^2}\Bigl(1+\frac2\theta\Bigr)\log^2\frac1\varepsilon ;
    \]
    for recipe~C$^2$, $m$ is given by Lemma~\ref{lem:sinc} at $(\delta_0,\Lambda)=(\eta_0/(10B),(20B/\eta_0)^2)$, so $m=O(\log^2(B/\eta_0))$.
    \item \emph{Internal scales, recipe~W:} $\sqrt\Lambda=(60B/\eta_0)^{1/\theta}$ and $t^{-1/2}=\sqrt{40}\,(20B/\eta_0)^{1/\theta}$, so that $t^{-1/2}/\sqrt\Lambda=\sqrt{40}\,3^{-1/\theta}>1$ if and only if $\theta>\log3/\log\sqrt{40}\approx0.596$, and $\tau=80t$.
    These are internal scales of the construction, paid for in the proofs by the tail term $2\Lambda^{-1/2}$ of Lemma~\ref{lem:HR} and the heat-time terms; they are not a requirement on a discretisation (Section~\ref{subsec:discretised}).
    \item \emph{Smallest heat time, recipe~W:} $\me^{-Mh}\in(\delta_0/(8\Lambda),\,\me\,\delta_0/(8\Lambda)]$, so $\me^{-Mh}\asymp\delta_0/\Lambda\asymp(\eta_0/B)^{1+2/\theta}$.
\end{enumerate}
\end{proposition}

\begin{table}[ht]
    \centering
    \caption{Counts and magnitudes of the scheme $N_\varepsilon$, with $\theta$ fixed. Rows (a)--(d) are polylogarithmic in $1/\varepsilon$; rows (e)--(g) are polynomial in $1/\varepsilon$; row (h) is $O(1/\varepsilon)$ up to the logarithm in $N$, the class constants and, for recipe~C$^2$, a factor $\sqrt d$, and row (i) is exponential in $d$ through $M^d$. Entries for recipe~C$^2$ are on $\TT^d$ with $K=0$. In the generic case $\eta_0=\varepsilon\bar a(1-q)^2/8=\varepsilon a_{\min}/(4(\kappa+1))$ one has $B/\eta_0=4(\kappa+1)B/(\varepsilon a_{\min})$. Rows (e)--(f) are internal scales of the construction, not frequencies a grid must resolve (Section~\ref{subsec:discretised}).}
    \label{tab:counts}
    \small
    \renewcommand{\arraystretch}{1.2}
    \begin{tabular}{>{\raggedright\arraybackslash}p{0.27\textwidth}>{\raggedright\arraybackslash}p{0.34\textwidth}>{\raggedright\arraybackslash}p{0.31\textwidth}}
        \toprule
        Quantity & Heat family & Mixed family \\
        \midrule
        (a) theorem-prescribed scalars, shared across blocks
        & $2m+4$ ($m$ sinc weights, $m$ sinc times, $t,\tau,T,\bar a$)
        & $5$ scalars $(s,t,\tau,T,\bar a)$ and the integer $N$ \\
        (b) operator applications (Definition~\ref{def:counting})
        & $(N+1)(m+3)$ heat
        & $3N+3$ heat, $N+1$ resolvent \\
        (c) pointwise products
        & $3N$
        & $3N$ \\
        (d) depth
        & $N=O(\log(1/\varepsilon)/\log(1/q))$
        & same \\
        (e) internal scales, recipe~W
        & $\sqrt\Lambda=(60B/\eta_0)^{1/\theta}$, $t^{-1/2}=\sqrt{40}(20B/\eta_0)^{1/\theta}$
        & $t^{-1/2}$ only (no band) \\
        (f) internal scales, recipe~C$^2$
        & $\sqrt\Lambda=20B/\eta_0$, $t^{-1/2}=5B/\eta_0$
        & $t^{-1/2}=5B/\eta_0$, $\tau^{-1/2}=7B/\eta_0$ \\
        (g) largest multiplier $1/(2t)$
        & W: $20(20B/\eta_0)^{2/\theta}$;\ \ C$^2$: $12.5\,B^2/\eta_0^2$
        & same \\
        (h) grid size $M$ per direction of the exact realisation with band-limited inputs (Corollary~\ref{cor:bandlimited})
        & $M=2(D_f+ND_a)+1$ with $D_f+1\ge1/(\varepsilon a_{\min})$ and, for the Jackson-smoothed coefficient, $D_a=2n-2$, $n\asymp\sqrt d\,\bar aB/(\varepsilon a_{\min}^2)$ (C$^2$), $n\asymp(\bar aB/(\varepsilon a_{\min}^2))^{1/(1+\theta)}$ (W)
        & same \\
        (i) arithmetic cost of that realisation
        & $O\bigl((N+1)(m+3)M^d\log M+3NM^d\bigr)$
        & $O\bigl((4N+4)M^d\log M+3NM^d\bigr)$ \\
        \bottomrule
    \end{tabular}
\end{table}

\begin{proof}
\emph{Table.}
Per block: $P_t(bv)$ and $P_tv$ (two applications at time $t$), $S$ applied to one channel ($m$ applications, or one resolvent), and $I-P_T$ applied once to the merged channel; the precomputation $P_tb$, $P_\tau b$, $v_0=(I-P_T)Sf$ costs $m+3$ applications (resp.\ $3$ heat and $1$ resolvent); the products per block are $bv$, $P_tb\cdot P_tv$ and $v\,b_\tau$.
The sinc nodes $(h\me^{jh},\me^{jh})$, $-M\le j\le\Nsinc$, are determined by $(h,M,\Nsinc)$; in the mixed family the scheme is determined by $(s,t,\tau,T,\bar a)$ and $N$.
Rows (e)--(g) evaluate Table~\ref{tab:recipes}: for recipe~W, $\sqrt\Lambda=(60B/\eta_0)^{1/\theta}$ and $t^{-1/2}=\sqrt{40}(20B/\eta_0)^{1/\theta}$; for recipe~C$^2$ with $K=0$, $\sqrt\Lambda=20B/\eta_0$, $t^{-1/2}=5B/\eta_0$, $\tau^{-1/2}=7B/\eta_0$; rows (h) and (i) are Corollary~\ref{cor:bandlimited} together with the FFT cost $O(M^d\log M)$ of one multiplier application and $O(M^d)$ of one product on $M^d$ grid points.
\emph{(a)} Since $\eta_0\le\varepsilon\bar a(1-q)^2/8$, $60B/\eta_0\ge480B/(\varepsilon\bar a(1-q)^2)\ge4/(\varepsilon\bar a(1-q))$.
\emph{(b)} From Lemma~\ref{lem:sinc}, $(M+1)h<\log(8\Lambda/\delta_0)+h$ and $(\Nsinc+1)h<\log(2\log\frac8{h\delta_0})+h$, so $m\le h^{-1}[\log(8\Lambda/\delta_0)+\log(2\log\frac8{h\delta_0})]+1$; for recipe~W, $h^{-1}=1+\frac2{\pi^2}\log\frac1{\delta_0}\le1+\frac2{\pi^2}L_\varepsilon$ and $\log\frac{8\Lambda}{\delta_0}=\log8+\frac2\theta\log\frac{60B}{\eta_0}+\log\frac{15B}{\eta_0}\le\log8+(1+\frac2\theta)L_\varepsilon$.
\emph{(c)} The ratio $t^{-1/2}/\sqrt\Lambda=\sqrt{40}\,3^{-1/\theta}$ follows from the two expressions above, and $\tau=80t$ from Table~\ref{tab:recipes}.
\emph{(d)} From $M=\lceil\log(8\Lambda/\delta_0)/h\rceil-1$, $\log(8\Lambda/\delta_0)-h\le Mh<\log(8\Lambda/\delta_0)$, and $h\le1$.
\end{proof}

The mixed family shows that the $O(\theta^{-1}\log^2(B/\eta_0))$ heat applications per block of the heat family are the price of representing $1/\lambda$ by exponential sums, not a feature of the mechanism: the resolvent replaces the approximate inverse, while the coupling $\nabla a\cdot\nabla u$ is still produced by heat layers, and the internal scales are those of the heat layers, not of $s$.
The five scalars of the mixed family are theorem-prescribed quantities determined by the class parameters and the tolerance; the class $\Fmix$ is problem-adapted, and their number is not comparable with lower bounds on parametric complexity for generic architectures \cite{LS23}.

\begin{remark}[The two recipes compared]
\label{rem:recipe-comparison}
For $\theta=1$ and $K=0$, recipe~C$^2$ improves on recipe~W in every magnitude.
At equal $B$, the band radius is $\sqrt\Lambda=20B/\eta_0$ instead of $60B/\eta_0$, and $t^{-1/2}=5B/\eta_0$ instead of $\sqrt{40}\cdot20B/\eta_0\approx126B/\eta_0$, so that the larger of the two internal scales is smaller by a factor of about $6$ and the multiplier $1/(2t)$ by a factor $640$.
These factors compare the two recipes at the same value of $B$, that is, on different classes; on the same class, the Wiener ball $\{\norm{b}_{A^2}\le B\}$, which lies in the class of Corollary~\ref{cor:torus-C2} with $2B$ in place of $B$, they are about $1.5$ for the band radius, $12.6$ for $t^{-1/2}$, $3.2$ for the larger internal scale and $160$ for $1/(2t)$.
The numerical consequence is visible in Table~\ref{tab:numerics}.
\end{remark}

\subsection{Discretisation: band-limited inputs}
\label{subsec:discretised}
The theorems of Sections~\ref{sec:dirichlet} and~\ref{sec:torus} involve no grid, and three scales have to be kept apart: the band $\sqrt\Lambda$ of the approximate inverse and the heat scale $t^{-1/2}$ of the product defect, which are \emph{internal} to the construction and are paid for in the proofs by the tail term $2\Lambda^{-1/2}$ of Lemma~\ref{lem:HR} and the heat-time terms; the resolution at which data and coefficient are \emph{represented}; and the grid on which the architecture is \emph{evaluated}.
A realisation that evaluates the symbols $r(\abs k^2)$ and $\me^{-t\abs k^2}$ exactly at the frequencies it represents has nothing to resolve at the internal scales, and for band-limited inputs the following corollary settles the other two: the scheme is then a finite computation in exact arithmetic, with error equal to the continuum error plus the representation errors of the inputs.

For $D\in\NN_0$ let $Q_D$ denote the Fourier projection onto the cube $\{k\in\ZZ^d:\abs k_\infty\le D\}$, and say that a trigonometric polynomial has \emph{degree} at most $D$ if its spectrum lies in this cube; the degree to which the data are truncated is denoted by $D_f$.
For $n\ge1$ let $J_n$ be the Jackson mean of order $n$ on $\TT^d$ (Lemma~\ref{lem:jackson}): $(J_nb)^\wedge_\mu=\hat J_n(\mu)\hat b_\mu$, where $J_n\ge0$ has mass one, $0\le\hat J_n(\mu)\le1$, and $\hat J_n(\mu)=0$ for $\abs\mu_\infty>2n-2$.

\begin{corollary}[Exact pseudo-spectral realisation with band-limited inputs]
\label{cor:bandlimited}
Let $\mathcal A$ be the class of Corollary~\ref{cor:torus-C2} or the class $\Acl$ of Theorem~\ref{thm:main}, and let $N_\varepsilon$ be the corresponding scheme, in either family, with the parameters of Table~\ref{tab:recipes} unchanged.
Let $a\in\mathcal A$, let $\tilde a\in\mathcal A$ be a trigonometric polynomial of degree at most $D_a$, and for $D_f\ge0$ let $\tilde f:=Q_{D_f}f$.
\begin{enumerate}[label=(\roman*)]
    \item Every function formed by the scheme $N_\varepsilon(\tilde a,\tilde f)$---iterates, products and layer outputs---is a trigonometric polynomial of degree at most $D:=D_f+ND_a$.
    Hence on the uniform grid with $M\ge2D+1$ points per direction, the realisation of $N_\varepsilon$ that applies the layers as Fourier multipliers to grid values and forms the products pointwise on the grid returns the grid values of $N_\varepsilon(\tilde a,\tilde f)$ exactly, in exact arithmetic.
    \item For mean-zero $f\in L^2(\TT^d)$,
    \begin{equation}
        \label{eq:bandlimited}
        \norm{G(a,f)-N_\varepsilon(\tilde a,\tilde f)}_1\le\Bigl[\varepsilon+\frac1{a_{\min}(D_f+1)}+\frac{\norm{a-\tilde a}_\infty}{a_{\min}^2}\Bigr]\norm f .
    \end{equation}
    \item For $n\ge1$, $\tilde a:=\bar a\,(1+J_nb)$ with $b=a/\bar a-1$ and the Jackson mean $J_n$ belongs to $\mathcal A$ (with the same $a_{\min}$, $a_{\max}$ and $B$), has degree at most $D_a=2n-2$, and satisfies $\norm{a-\tilde a}_\infty=\bar a\norm{b-J_nb}_\infty\le\pi^{3/2}\sqrt d\,\bar aB/n$ for the class of Corollary~\ref{cor:torus-C2} and $\norm{a-\tilde a}_\infty\le(\pi^3/2)^{(1+\theta)/2}\bar aB\,n^{-(1+\theta)}$ for $\Acl$.
\end{enumerate}
In particular, with this $\tilde a$, $D_f+1\ge1/(\varepsilon a_{\min})$ together with $n\ge\pi^{3/2}\sqrt d\,\bar aB/(\varepsilon a_{\min}^2)$ for the class of Corollary~\ref{cor:torus-C2}, or $n\ge(\pi^3/2)^{1/2}\bigl(\bar aB/(\varepsilon a_{\min}^2)\bigr)^{1/(1+\theta)}$ for $\Acl$, gives $\norm{G(a,f)-N_\varepsilon(\tilde a,\tilde f)}_1\le3\varepsilon\norm f$.
\end{corollary}

\begin{proof}
\emph{(i)} Fourier multipliers do not increase the degree, and the product of trigonometric polynomials of degrees $D_1$, $D_2$ has degree at most $D_1+D_2$.
Since $v_0=R_m\tilde f$ has degree at most $D_f$, and block $i$ forms the products $\tilde bv_i$, $P_t\tilde b\cdot P_tv_i$ and $v_i\tilde b_\tau$ (with $\tilde b=\tilde a/\bar a-1$) of degree at most $\deg v_i+D_a$ before applying multipliers, induction gives $\deg v_i\le D_f+iD_a$ and the bound $D$ for every function formed.
On a grid with $M\ge2D+1$ points per direction the grid values of a trigonometric polynomial of degree at most $D$ determine its coefficients exactly (no aliasing), and the grid values of a product are the products of the grid values; so each step of the realisation reproduces the exact function.

\emph{(ii)} Put $u=G(a,f)$, $u'=G(\tilde a,f)$ and $u''=G(\tilde a,\tilde f)$.
Subtracting the weak formulations and testing with $u-u'$ gives $a_{\min}\norm{u-u'}_1^2\le\norm{a-\tilde a}_\infty\norm{u'}_1\norm{u-u'}_1$, and $\norm{u'}_1\le\norm f_{-1}/a_{\min}$ since $\tilde a\ge a_{\min}$, so $\norm{u-u'}_1\le\norm{a-\tilde a}_\infty\norm f/a_{\min}^2$.
Next, $\norm{u'-u''}_1\le\norm{(I-Q_{D_f})f}_{-1}/a_{\min}\le\norm f/(a_{\min}(D_f+1))$, because $\abs k\ge\abs k_\infty\ge D_f+1$ off the cube.
Finally $\norm{u''-N_\varepsilon(\tilde a,\tilde f)}_1\le\varepsilon\norm{\tilde f}\le\varepsilon\norm f$ by Corollary~\ref{cor:torus-C2} or Theorem~\ref{thm:main}, which apply to $\tilde a\in\mathcal A$.

\emph{(iii)} $J_n$ is convolution with a nonnegative kernel of mass one (Lemma~\ref{lem:jackson}(i)), so it preserves pointwise bounds: $-q\le J_nb\le q$, hence $a_{\min}\le\tilde a\le a_{\max}$.
It commutes with $\nabla$ and $\Delta$ and is an $L^\infty$-contraction, so $\norm{\nabla J_nb}_\infty+\norm{\Delta J_nb}_\infty\le\norm{\nabla b}_\infty+\norm{\Delta b}_\infty$, and $0\le\hat J_n\le1$ gives $\norm{J_nb}_{A^{1+\theta}}\le\norm b_{A^{1+\theta}}$; the degree is that of the kernel, and the bounds on $\norm{b-J_nb}_\infty$ are Lemma~\ref{lem:jackson}(iii) and (iv).
\end{proof}

Three comments.
First, the data term of \eqref{eq:bandlimited} has the right order: for $a\equiv a_{\min}$ (that is, $b\equiv-q$, which belongs to both classes) and $f$ concentrated at the frequency $(D_f+1,0,\dots,0)$, $\norm{(I-Q_{D_f})G(a,f)}_1=\norm f/(a_{\min}(D_f+1))$, so a representation degree $D_f+1\ge1/(\varepsilon a_{\min})$ is necessary for the truncation error of $L^2$ data to be at most $\varepsilon\norm f$, and neither $\sqrt\Lambda$ nor $t^{-1/2}$ enters it.
The coefficient costs a degree $D_a$ per block, so that $M=2D_f+2ND_a+1$: with the Jackson mean of (iii), for the class of Corollary~\ref{cor:torus-C2} the coefficient part $2N(2n-2)$ exceeds $2D_f\asymp2/(\varepsilon a_{\min})$ by a factor of order $N\sqrt d\,\bar aB/a_{\min}$, for $\Acl$ it is $o(D_f)$ as $\varepsilon\to0$, and in both cases $M=O(1/\varepsilon)$ up to the logarithm in $N$ and the class constants, with the additional factor $\sqrt d$ in the first case.
Second, the arithmetic cost on $M^d$ grid points is row (i) of Table~\ref{tab:counts}, $\tilde O(\varepsilon^{-d})$ for both recipes at fixed $\theta$; an FFT-preconditioned conjugate-gradient solve on the same grid costs $\tilde O(\sqrt\kappa\log(1/\varepsilon)M^d)$, and the Fourier--Galerkin emulation of \cite[Setting~25, Theorem~26]{KLM21}, for right-hand sides in $\dot H^{k-1}$ and coefficients in $H^s$ with $s\ge d/2+k+\delta$, reaches error $\varepsilon$ with $N\asymp\varepsilon^{-1/k}$ modes per direction and size $\lesssim N^d\log N$, per-mode weights counted as parameters \cite[Definition~12 and Remark~27]{KLM21}; it trades data regularity for resolution, whereas here the data are only in $L^2$ and the coefficient regularity enters through $n$, and the two parameter counts are not comparable (the depths are).
Third, the corollary keeps the growing spectral support of the iterates; the variant that projects each iterate back to degree $D_f$---a de-aliased Fourier--Galerkin Richardson iteration---would need a bound on the iterates in a norm stronger than $\Hd$ and is not treated here (Section~\ref{sec:conclusion}).

\subsection{Floating point}
\label{subsec:floating}
The parameters of Table~\ref{tab:recipes} are chosen for the proofs, and they make the scheme ill-conditioned at small $\varepsilon$.

\begin{remark}[Floating point]
\label{rem:floating}
Coefficient magnitudes are polynomial in $B\kappa^2/(\varepsilon\bar a)$ with exponent $2/\theta$ (row (g) of Table~\ref{tab:counts}), and evaluating $\Gamma_t$ as $\frac1{2t}[P_t(bv)-P_tb\,P_tv]$ loses about $\log_{10}(1/(2t))$ digits through cancellation.
With the proved constants the construction is therefore realisable in double precision only for moderate $\varepsilon$; the threshold depends on the recipe, since
\begin{align*}
    \frac1{2t}&=8000\,\frac{B^2}{\eta_0^2}&&\text{(recipe~W at $\theta=1$)},\\
    \frac1{2t}&=12.5\,\frac{B^2}{\eta_0^2}&&\text{(recipe~C$^2$, $K=0$)}.
\end{align*}
Section~\ref{sec:numerics} reports empirical thresholds for one example and both recipes, which are observations, not theorems.
The cancellation is intrinsic to synthesising a derivative from a semigroup at scale $\sqrt t\asymp\eta_0/B$: the same smallness of $t$ that makes the product defect accurate (Theorem~\ref{thm:gamma-consistency}, Remark~\ref{rem:sharpness}) is what the division by $2t$ amplifies.
The resolvent removes the smallest exponential-sum times but not $1/(2t)$, and supplying $L_0a$ as an input channel removes the $\tau$-channel but not $\Gamma_t$.
\end{remark}

\section{Numerical sanity check}
\label{sec:numerics}

We check the construction with the parameters exactly as prescribed, for both recipes and both families; the experiment is not an assessment of computational performance, and its purpose is to see the mechanism run and to locate the floating-point threshold of Remark~\ref{rem:floating}.
A second experiment checks the data term of Corollary~\ref{cor:bandlimited}.
We take $d=1$, $a(x)=1+0.6\sin x+0.25\cos(2x+1)$ and $f\propto\cos x+0.5\sin3x+0.3\cos4x$ with $\norm{f}_{L^2}=1$.
On the grid, $a_{\min}\approx0.196$, $a_{\max}\approx1.644$, $\bar a\approx0.920$ and $q\approx0.787$ ($\kappa\approx8.4$); the class constants are $B=\norm{b}_{A^2}\approx5.14$ for recipe~W with $\theta=1$ and $B=\norm{b'}_\infty+\norm{b''}_\infty\approx2.84$ for recipe~C$^2$ with $K=0$ (Corollary~\ref{cor:torus-C2}); $\eta_0$ and $N$ are the same for both.
The reference solution is computed by exact one-dimensional quadrature.
Layers are applied as Fourier multipliers by the FFT on a uniform grid with $n=2^{15}$ points, in double precision.
All parameters are as in Table~\ref{tab:recipes} and Lemma~\ref{lem:sinc}, each recipe with its own $(\delta_0,\Lambda,t,\tau)$.

\begin{table}[ht]
    \centering
    \caption{Both prescribed recipes in double precision on the grid $n=2^{15}$: $\Hd$ error divided by the target $\varepsilon$ for the heat and the mixed family, with the common depth $N$ and, for each recipe, the number $m$ of exponential-sum times and the largest multiplier $1/(2t)$. The operator counts are $(N+1)(m+3)$ heat applications (heat family) and $3N+3$ heat plus $N+1$ resolvent applications (mixed family), Definition~\ref{def:counting}.}
    \label{tab:numerics}
    \footnotesize
    \renewcommand{\arraystretch}{1.25}
    \setlength{\tabcolsep}{4pt}
    \begin{tabular}{cc|cccc|cccc}
        & & \multicolumn{4}{c|}{recipe~W ($\theta=1$, $B\approx5.14$)} & \multicolumn{4}{c}{recipe~C$^2$ ($K=0$, $B\approx2.84$)}\\
        $\varepsilon$ & $N$ & $m$ & $1/(2t)$ & heat & mixed & $m$ & $1/(2t)$ & heat & mixed\\
        \hline
        $0.5$    & 16 & 121 & $3.1\cdot10^{10}$ & 0.011 & 0.011 & 100 & $1.5\cdot10^{7}$  & 0.011 & 0.011\\
        $0.2$    & 20 & 137 & $1.9\cdot10^{11}$ & 0.009 & 0.009 & 115 & $9.2\cdot10^{7}$  & 0.009 & 0.009\\
        $0.1$    & 23 & 150 & $7.7\cdot10^{11}$ & 0.009 & 0.009 & 127 & $3.7\cdot10^{8}$  & 0.008 & 0.008\\
        $0.05$   & 26 & 164 & $3.1\cdot10^{12}$ & 0.012 & 0.013 & 140 & $1.5\cdot10^{9}$  & 0.007 & 0.007\\
        $0.02$   & 29 & 183 & $1.9\cdot10^{13}$ & 0.10  & 0.10  & 157 & $9.2\cdot10^{9}$  & 0.008 & 0.008\\
        $0.01$   & 32 & 198 & $7.7\cdot10^{13}$ & 0.74  & 0.95  & 171 & $3.7\cdot10^{10}$ & 0.007 & 0.007\\
        $0.005$  & 35 & 214 & $3.1\cdot10^{14}$ & 5.3   & 6.8   & 186 & $1.5\cdot10^{11}$ & 0.006 & 0.006\\
        $0.002$  & 39 & 235 & $1.9\cdot10^{15}$ & $1.4\cdot10^{2}$ & $1.5\cdot10^{2}$ & 206 & $9.2\cdot10^{11}$ & 0.063 & 0.059\\
        $0.001$  & 42 & 253 & $7.7\cdot10^{15}$ & $2.4\cdot10^{2}$ & $3.5\cdot10^{2}$ & 223 & $3.7\cdot10^{12}$ & 0.98  & 0.92
    \end{tabular}
\end{table}

\emph{Moderate $\varepsilon$.}
For $\varepsilon\ge0.05$ the observed error is about two orders of magnitude below the target in both families and for both recipes (Table~\ref{tab:numerics}), consistent with pessimistic constants; the two families agree although the mixed family uses one resolvent instead of $m$ exponential-sum times per block, and the counts match Table~\ref{tab:counts} (e.g.\ $(N+1)(m+3)=24\cdot153=3672$, resp.\ $24\cdot130=3120$, heat applications and $3N+3=72$ at $\varepsilon=0.1$).

\smallskip
\emph{Small $\varepsilon$.}
With recipe~W, accuracy deteriorates from $\varepsilon\approx0.02$ on; at $\varepsilon=0.01$ the target is still met, at $\varepsilon=0.005$ it is missed by a factor of about $5$ (heat) and $7$ (mixed), and below that the output is dominated by rounding.
With recipe~C$^2$ the error stays about two orders of magnitude below the target down to $\varepsilon=0.005$, and the target is met at $\varepsilon=0.002$ and, within rounding, at $\varepsilon=0.001$ (err$/\varepsilon\approx0.98$ and $0.92$).
The empirical double-precision thresholds for this example are therefore $\varepsilon\approx0.005$--$0.01$ for recipe~W and $\varepsilon\approx0.001$ for recipe~C$^2$.
The mechanism is the cancellation of Remark~\ref{rem:floating}: at $\varepsilon=0.01$ recipe~W has $1/(2t)\approx7.7\cdot10^{13}$, and machine precision times $1/(2t)$ times $O(1)$ data is of order $10^{-2}$; recipe~C$^2$ has $1/(2t)$ smaller by a factor $640$ at equal $B$, and here moreover a smaller $B$, so the same loss occurs about one order of magnitude further down in $\varepsilon$.
Because the rounding-dominated entries depend on the order of the floating-point operations, re-implementations differ there by up to about $30\%$, with the same qualitative conclusion.

\smallskip
\emph{Exactness on the grid.}
In this experiment the coefficient has degree $2$ and the data degree $4$, so every function formed has degree at most $4+2N\le88$ (Corollary~\ref{cor:bandlimited}(i)): the computation is exact in exact arithmetic on any grid with more than $177$ points, and the grid size matters only in the rounding-dominated regime, where every represented mode contributes noise that $1/(2t)$ amplifies.
The internal scales $\sqrt\Lambda$ and $t^{-1/2}$ of Table~\ref{tab:counts}, here between $2\cdot10^4$ and $1.2\cdot10^8$, are not resolved and need not be; the experiment tests the constants of the theorems and the floating-point threshold, not resolution.

\smallskip
\emph{Band-limited inputs.}
To see the data term of Corollary~\ref{cor:bandlimited}, we take a coefficient that is itself a trigonometric polynomial of degree $6$, $a(x)=1+0.5\sin x+0.25\cos(2x+1)+0.1\cos5x+0.1\sin6x$ ($a_{\min}\approx0.265$, $a_{\max}\approx1.740$, $\kappa\approx6.6$, $B=\norm{b'}_\infty+\norm{b''}_\infty\approx8.25$), so that $\tilde a=a$ and the coefficient term vanishes, and rough data $f(x)\propto x-\pi$ on $(0,2\pi)$, normalised in $L^2$ (so that $\hat f_k\propto1/k$), truncated to $\tilde f=Q_{D_f}f$.
Recipe~C$^2$ is run on the grid with $M$ points, $M$ the smallest power of two with $M\ge2(D_f+6N)+1$, and the error is measured against the solution for the untruncated $f$ (Table~\ref{tab:bandlimited}).
The error decreases with $D_f$, stays below the bound $\varepsilon+1/(a_{\min}(D_f+1))$ of \eqref{eq:bandlimited}, and for this datum reaches the floor set by the continuum error once $D_f$ is a few times $1/(\varepsilon a_{\min})\approx38$ (resp.\ $189$); since $\hat f_k\propto1/k$, the truncation error of the sawtooth decays like $D_f^{-3/2}$, so the table shows that the bound holds, not the worst-case rate $D_f^{-1}$ of the data term. Meanwhile the internal scales are $\sqrt\Lambda\approx1.9\cdot10^5$ and $9.4\cdot10^5$; running the same computation on $2^{16}$ points changes no entry by more than $2\%$, as part (i) of the corollary predicts.

\begin{table}[ht]
    \centering
    \caption{Band-limited inputs (recipe~C$^2$, $d=1$, coefficient of degree $6$, sawtooth data truncated to degree $D_f$): $\Hd$ error of $N_\varepsilon(a,Q_{D_f}f)$ against $G(a,f)$, divided by $\varepsilon$, on the grid with $M$ points, next to the data term $1/(a_{\min}(D_f+1))$ of \eqref{eq:bandlimited} divided by $\varepsilon$.}
    \label{tab:bandlimited}
    \footnotesize
    \renewcommand{\arraystretch}{1.25}
    \setlength{\tabcolsep}{5pt}
    \begin{tabular}{c|ccc|ccc}
        & \multicolumn{3}{c|}{$\varepsilon=0.1$, $N=17$} & \multicolumn{3}{c}{$\varepsilon=0.02$, $N=22$}\\
        $D_f$ & $M$ & err$/\varepsilon$ & $1/(\varepsilon a_{\min}(D_f+1))$ & $M$ & err$/\varepsilon$ & $1/(\varepsilon a_{\min}(D_f+1))$\\
        \hline
        $2$   & $256$  & $1.31$   & $12.6$  & $512$  & $6.54$   & $62.9$\\
        $8$   & $256$  & $0.178$  & $4.19$  & $512$  & $0.899$  & $21.0$\\
        $32$  & $512$  & $0.0219$ & $1.14$  & $512$  & $0.106$  & $5.72$\\
        $128$ & $512$  & $0.0063$ & $0.293$ & $1024$ & $0.0138$ & $1.46$\\
        $512$ & $2048$ & $0.0058$ & $0.074$ & $2048$ & $0.0056$ & $0.368$
    \end{tabular}
\end{table}

These numbers are observations for two examples, not theorems.
They show that the ill-conditioning at small $\varepsilon$ seen with recipe~W is a property of that recipe's constants, not of the mechanism, which recipe~C$^2$ realises with the same depth and smaller multipliers---and that even recipe~C$^2$ reaches the limit of double precision at $\varepsilon\approx10^{-3}$; and that the grid a realisation needs is set by the inputs, as Corollary~\ref{cor:bandlimited} says, not by the internal scales of the construction.

\section{Conclusions and limitations}
\label{sec:conclusion}

The Darcy solution operator can be emulated, with explicit operator-norm control in $\Hd$, by a scheme built only from functions of one fixed reference operator and pointwise products.
What makes this possible is that a heat layer applied to a product, minus the product of heat layers, is a gradient--gradient coupling: the product defect of one heat layer is a first-order-accurate derivative surrogate, and the Laplacian-preconditioned Richardson iteration needs exactly that coupling and nothing else that is not a multiplication.
Because this is a statement of carr\'e du champ calculus, it holds on any strongly local Dirichlet space with a spectral gap and a Bakry--\'Emery curvature lower bound (Theorem~B), with constants that on $\TT^d$ improve on the Fourier analysis for $\theta=1$.
Strong locality is used through the Leibniz rule for $\Gamma$; graph Laplacians, for which a carr\'e du champ exists but no Leibniz rule holds, are therefore not covered, and whether an analogous construction exists for them is open.
Architectures built on the spectral decomposition of a fixed Laplacian have been proposed for operator learning on manifolds \cite{CLM24} and for reaction--diffusion dynamics \cite{WH25}, with approximation theory for a reaction--diffusion system in \cite{FOW26}; they expand in eigenfunctions where the present scheme applies the heat semigroup.

The results are approximation theorems for continuum operators with prescribed parameters, not statements about learned models or efficient solvers, and they concern scalar, uniformly positive coefficients with bounded $\nabla b$ and $L_0b$ (or $1+\theta$ Wiener derivatives), with all parameters polynomial in the contrast $\kappa$ (Section~\ref{subsec:scope}).
One limitation is intrinsic to the product-defect block: the scheme is ill-conditioned with the proved constants, since extracting a derivative at scale $\sqrt t\asymp\eta_0/B$---a scale forced by the product-defect term, Remark~\ref{rem:sharpness}---multiplies by $1/(2t)$, so the scheme is realisable in double precision only down to $\varepsilon\approx10^{-3}$ on the test problem, and replacing the exponential sums by a resolvent does not remove this cancellation.
Discretisation is not such a limitation: the internal scales of the construction are not frequencies a grid must resolve, and on the torus the scheme applied to band-limited inputs is evaluated exactly on a grid set by the data and the coefficient (Corollary~\ref{cor:bandlimited}).

Several directions remain open: the variant of Corollary~\ref{cor:bandlimited} that projects each iterate back to degree $D_f\asymp1/(\varepsilon a_{\min})$, a de-aliased Fourier--Galerkin Richardson iteration, which needs a stability bound for the iterates in a norm stronger than $\Hd$; discretisations on manifolds and Dirichlet spaces, where the heat layers are replaced by the semigroup of a mesh Laplacian; the behaviour of the construction under perturbations of the reference generator, which is relevant for geometry-dependent operator learning; trainable variants of the class, for instance with learned diffusion times and pointwise maps; and Chebyshev acceleration: since $I+K_b$ is self-adjoint on $\Hd$ with spectrum in $[1-q,1+q]$, the exact Chebyshev semi-iteration needs $O(\sqrt\kappa\log(1/\varepsilon))$ steps with block-dependent scalar weights and stays inside the class, but the version with the inexact step $\Knet$ has not been analysed.

\appendix

\section{Proof of Lemma~\ref{lem:sinc}}
\label{app:sinc}

For $h>0$ and integers $M,\Nsinc\ge0$ write $s_h(\lambda):=h\sum_{j=-M}^{\Nsinc}\me^{jh}\me^{-\lambda\me^{jh}}$, so that $r(\lambda)=s_h(\lambda)(1-\me^{-T\lambda})$.

\emph{Step 1 (Poisson bound).}
Let $\rho(h):=13\,h^{-1/2}\me^{-\pi^2/h}$.
We claim that for $0<h\le1$ and all $\lambda>0$,
\[
    \Bigl|\lambda h\sum_{j\in\ZZ}\me^{jh}\me^{-\lambda\me^{jh}}-1\Bigr|\le\rho(h).
\]
Put $\phi(s)=\exp(s-\me^s)$; then $\phi>0$, $\int_\RR\phi=\int_0^\infty\me^{-x}\dint x=1$ (with $x=\me^s$), and $\lambda\me^s\me^{-\lambda\me^s}=\phi(s+\log\lambda)$.
Hence $\lambda h\sum_j\me^{jh}\me^{-\lambda\me^{jh}}=h\sum_jg(jh)$ with $g(s)=\phi(s+\sigma)$, $\sigma=\log\lambda$.
The function $g$ is smooth, $g(s)=O(\me^{-\abs{s}})$ as $s\to-\infty$ and $O(\me^{-\me^s/2})$ as $s\to+\infty$, and $\hat g(\omega)=\me^{i\omega\sigma}\hat\phi(\omega)$ with
\[
    \hat\phi(\omega)=\int_0^\infty\me^{-x}x^{-i\omega}\dint x=\Gamma(1-i\omega),
    \qquad\text{so that}\qquad
    \abs{\hat g(\omega)}=\Bigl(\frac{\pi\omega}{\sinh\pi\omega}\Bigr)^{1/2}.
\]
Poisson summation gives $h\sum_jg(jh)=1+\sum_{n\ne0}\hat g(2\pi n/h)$.
With $x_n=2\pi^2n/h\ge2\pi^2$ one has $\sinh x_n\ge0.4999\,\me^{x_n}$, so
\[
    \abs{\hat g(2\pi n/h)}\le(4.0008\pi^2n/h)^{1/2}\me^{-\pi^2n/h}.
\]
Summing over $n\ne0$ with $\varrho=\me^{-\pi^2/h}\le\me^{-\pi^2}<6\cdot10^{-5}$ and $\sum_{n\ge1}\sqrt n\varrho^n\le1.0002\varrho$ gives the bound $(4.0012\pi/\sqrt h)\me^{-\pi^2/h}\le\rho(h)$.

\smallskip
\emph{Step 2 (truncation).}
For $0<h\le1$ and $\lambda\ge1$ we have $0\le\lambda r(\lambda)\le1+\rho(h)$, and for $1\le\lambda\le\Lambda$
\[
    \lambda r(\lambda)\ge1-\rho(h)-2\Lambda\me^{-(M+1)h}-\bigl(hX+1\bigr)\me^{-X}-(1+\rho(h))\me^{-T},\qquad X:=\me^{(\Nsinc+1)h}.
\]
All terms of the full sum are positive, so dropping the tails $j>\Nsinc$ and $j<-M$ only decreases it; together with $0\le1-\me^{-T\lambda}\le1$ this gives the upper bound.

\emph{Lower tail.}
$h\sum_{j<-M}\lambda\me^{jh}\me^{-\lambda\me^{jh}}\le h\lambda\sum_{j\ge M+1}\me^{-jh}=\lambda\frac{h}{1-\me^{-h}}\me^{-(M+1)h}\le2\lambda\me^{-(M+1)h}$, since $h/(1-\me^{-h})<1.6$ for $h\le1$.

\emph{Upper tail.}
Put $x_*=\lambda\me^{(\Nsinc+1)h}\ge X\ge1$; the terms are $h\,u_j\me^{-u_j}$ with $u_j=x_*\me^{jh}$, $j\ge0$, and $G(s):=x_*\me^{sh}\me^{-x_*\me^{sh}}$ is decreasing on $[0,\infty)$, so $h\sum_{j\ge1}G(j)\le h\int_0^\infty G=\me^{-x_*}$.
Hence the upper tail is at most $hx_*\me^{-x_*}+\me^{-x_*}\le(hX+1)\me^{-X}$, because $x\mapsto(hx+1)\me^{-x}$ is decreasing for $x\ge1$.

\noindent Finally $\lambda r=\lambda s_h-\lambda s_h\me^{-T\lambda}\ge\lambda s_h-(1+\rho)\me^{-T}$ for $\lambda\ge1$.

\smallskip
\emph{Step 3 (parameter choice).}
With $h=\pi^2/(\pi^2+2\log(1/\delta_0))\le1$ we have $\me^{-\pi^2/h}=\me^{-\pi^2}\delta_0^2$, and
\[
    \rho(h)\le13\,\bigl(9.87+2\log\tfrac1{\delta_0}\bigr)^{1/2}\cdot1.65\cdot10^{-5}\,\delta_0^2\le\frac{\delta_0}4,
\]
because $\delta_0(9.87+2\log\frac1{\delta_0})^{1/2}\le3.2$ on $(0,1]$.
The choice of $M$ gives $(M+1)h\ge\log(8\Lambda/\delta_0)$, hence $2\Lambda\me^{-(M+1)h}\le\delta_0/4$.
The choice of $\Nsinc$ gives $X\ge2L'$ with $L'=\log\frac8{h\delta_0}$; then
\[
    (hX+1)\me^{-X}\le(2hL'+1)(h\delta_0/8)^2\le(2L'\delta_0+1)\frac{\delta_0}{64}\le\frac{\delta_0}4,
\]
because $L'\delta_0=\delta_0\log\frac8{\delta_0}+\delta_0\log\frac1h\le\log8+\frac2{\pi^2\me}\le2.2$ on $(0,1]$.
Finally $(1+\rho)\me^{-T}\le1.001\,\delta_0/5\le\delta_0/4$.
Summing, \HR{} holds with $\delta=\delta_0$.
The count follows from $M+1\le\log(8\Lambda/\delta_0)/h+1$, $\Nsinc+1\le\log(2L')/h+1$ and $1/h=1+\frac2{\pi^2}\log\frac1{\delta_0}$. \qed

\section{Auxiliary lemmas on Dirichlet spaces, the non-inclusion example, and Jackson means}
\label{app:aux}

Throughout, Assumption~\ref{ass:M} holds and the notation is that of Sections~\ref{sec:mechanism} and~\ref{sec:dirichlet}.

\begin{lemma}[Density]
\label{lem:density}
$V_\infty$ is dense in $V$ for the norm $\Ecal_1(u)^{1/2}=(\Ecal(u,u)+\norm{u}^2)^{1/2}$.
\end{lemma}

\begin{proof}
For $u\in V$ let $u_n=(-n\vee u)\wedge n$.
Since truncation is a normal contraction, $u_n\in V$ and $\Ecal(u_n)\le\Ecal(u)$ by the Markov property, and $u_n\to u$ in $L^2$.
Thus $(u_n)$ is bounded in $(V,\Ecal_1)$, every weak limit point equals $u$, and weak lower semicontinuity gives $\Ecal_1(u_n)\to\Ecal_1(u)$; weak convergence and convergence of norms give $u_n\to u$ in $(V,\Ecal_1)$.
\end{proof}

\begin{proof}[Proof of Lemma~\ref{lem:spectral}]
By the spectral theorem, $\norm{(P_r-I)u}^2=\int[(1-\me^{-r\lambda})\lambda^{-1/2}]^2\lambda\dint\inner{E_\lambda u,u}$ and $(1-\me^{-r\lambda})\lambda^{-1/2}=\sqrt r\,(1-\me^{-x^2})/x$ with $x=\sqrt{r\lambda}$; the supremum is attained where $\me^{x^2}=1+2x^2$, giving $c_0=0.638173\ldots$.
The second bound follows in the same way from $\norm{(P_s-I)g}_{-1}^2=\int_{\lambda>0}\lambda^{-1}(1-\me^{-s\lambda})^2\dint\inner{E_\lambda g,g}$, and the third from $\me^{-r\lambda}\le1$.
For the fourth, $\sup_{\lambda\ge1}\lambda^{1/2}\me^{-T\lambda}=\me^{-T}$ when $T\ge\frac12$.
\end{proof}

\begin{proof}[Proof of Lemma~\ref{lem:interpolation}]
Let $\varphi\in L^2\cap L^\infty$, $f_s=P_{t-s}b$, $g_s=P_{t-s}u$ and $A(s)=\int f_sg_s\,P_s\varphi\dint\mu$, $s\in[0,t]$.
For $s<t$, $f_s,g_s\in D(L_0)\cap L^\infty$ with $\partial_sf_s=L_0f_s$, $\partial_sg_s=L_0g_s$ in $L^2$, and $L_0f_s=P_{t-s}L_0b$ is bounded; since $f_s,g_s$ are uniformly bounded, $\partial_s(f_sg_s)=(L_0f_s)g_s+f_s(L_0g_s)$ in $L^2$.
With $\psi=P_s\varphi\in V_\infty$ and $\partial_sP_s\varphi=-L_0\psi$,
\[
    A'(s)=\int\bigl[(L_0f_s)g_s+f_s(L_0g_s)\bigr]\psi\dint\mu-\int f_sg_s\,L_0\psi\dint\mu .
\]
By (F1), (F2) and the definition of $L_0$,
\begin{align*}
    \int f_sg_sL_0\psi&=\Ecal(\psi,f_sg_s)=\int\bigl[f_s\Gamma(g_s,\psi)+g_s\Gamma(f_s,\psi)\bigr],\\
    \int(L_0f_s)g_s\psi&=\Ecal(f_s,g_s\psi)=\int\bigl[g_s\Gamma(f_s,\psi)+\psi\Gamma(f_s,g_s)\bigr],
\end{align*}
and similarly for the term with $L_0g_s$.
Everything cancels except
\[
    A'(s)=2\int\psi\,\Gamma(f_s,g_s)=2\int\varphi\,P_s\Gamma(f_s,g_s),
\]
by the symmetry of $P_s$ in the form $\int(P_sh)\varphi\dint\mu=\int h\,P_s\varphi\dint\mu$ for $h\in L^1$, $\varphi\in L^\infty$ \cite[p.~350]{AGS15}, applied to $h=\Gamma(f_s,g_s)\in L^1$.
$A$ is continuous on $[0,t]$, and $\abs{A'(s)}\le2\me^{K(t-s)}\norm{\sqrt{\Gamma(b)}}_\infty\norm{u}_1\norm{\varphi}$ by (F3), so $A(t)-A(0)=\int_0^tA'$.
Since $A(t)-A(0)=\int[P_t(bu)-P_tb\,P_tu]\varphi$ and $\varphi$ ranges over a dense set, the identity follows.
\end{proof}

\begin{lemma}[Integrated gradient bound]
\label{lem:bochner}
For $g\in L^\infty$, $r>0$ and $h\in V$,
\[
    \Bigl\|\Gamma\Bigl(\int_0^rP_\sigma g\dint\sigma,\,h\Bigr)\Bigr\|_{L^2}\le\me^{Kr}\sqrt{2r}\,\norm{g}_\infty\norm{h}_1 .
\]
\end{lemma}

\begin{proof}
The map $\sigma\mapsto P_\sigma g$ is continuous $(0,r]\to V$ with $\norm{P_\sigma g}_1\le(2\me\sigma)^{-1/2}\norm{g}$, so $\int_0^rP_\sigma g\dint\sigma$ is a Bochner integral in $(V,\Ecal_1)$ and coincides with the $L^2$-integral.
Since $f\mapsto\Gamma(f,h)$ is bounded linear $V\to L^1$,
\[
    \Gamma\Bigl(\int_0^rP_\sigma g\dint\sigma,h\Bigr)=\int_0^r\Gamma(P_\sigma g,h)\dint\sigma\quad\text{in }L^1.
\]
By (F1) and \eqref{eq:reverse-poincare}, $\norm{\Gamma(P_\sigma g,h)}_{L^2}\le\me^{K\sigma}(2\sigma)^{-1/2}\norm{g}_\infty\norm{h}_1$, which is integrable on $(0,r]$; hence the integral is also an $L^2$-Bochner integral, and $\int_0^r\me^{K\sigma}(2\sigma)^{-1/2}\dint\sigma\le\me^{Kr}\sqrt{2r}$.
\end{proof}

\begin{lemma}[Non-inclusion in Corollary~\ref{cor:torus-C2}]
\label{lem:rudin-shapiro}
Let $a_{\min}<a_{\max}$ and $B\ge1$.
The class of Corollary~\ref{cor:torus-C2} contains coefficients with arbitrarily large $\norm{b}_{A^2}$.
\end{lemma}

\begin{proof}
It suffices to take $d=1$ (and functions of $x_1$ in general).
Let $N\le R$, with $N$ a power of $2$, and let $(\epsilon_j)_{j<N}$ be the Rudin--Shapiro signs \cite{Sha51,Rud59}, so that $\abs{\sum_{j<N}\epsilon_j\me^{ijx}}\le\sqrt{2N}$; put $P(x)=N^{-1/2}\me^{iRx}\sum_{j<N}\epsilon_j\me^{ijx}$ and $b=\operatorname{Re}P/(4\sqrt2R^2)$, so that $\norm{b}_\infty\le1/(4R^2)$.
Bernstein's inequality $\norm{T'}_\infty\le n\norm{T}_\infty$ for trigonometric polynomials $T$ of degree at most $n$ \cite{Zyg02}, applied with $n=2R$, gives $\norm{b'}_\infty\le1/(2R)$ and $\norm{b''}_\infty\le1$, so $\norm{b}_{\mathcal B}\le2$.
On the other hand $\hat b_{\pm(R+j)}=\epsilon_j/(8\sqrt2\,N^{1/2}R^2)$ for $0\le j<N$, so counting only the frequencies $\mu=R+j$, $0\le j<N$, gives
\[
    \norm{b}_{A^2}\ge\sum_{j<N}(R+j)^2\abs{\hat b_{R+j}}\ge\frac{N\cdot R^2}{8\sqrt2\,N^{1/2}R^2}=\frac{\sqrt N}{8\sqrt2}
\]
(the frequencies $-(R+j)$ double this), which is unbounded as $N\to\infty$; multiplying $b$ by $B/2$ and taking $R$ large (so that $\norm{b}_\infty\le B/(8R^2)\le q$) gives coefficients in the class with unbounded $A^2$ norm.
\end{proof}

\begin{lemma}[Jackson means]
\label{lem:jackson}
For $n\ge1$ let $j_n(s):=\lambda_n^{-1}\bigl(\sin(ns/2)/\sin(s/2)\bigr)^4$ on $\TT=\RR/2\pi\ZZ$, with $\lambda_n:=(2n^3+n)/3$, let $J_n(y):=\prod_{i=1}^dj_n(y_i)$ on $\TT^d$, and let $J_nb:=J_n*b$ denote convolution with respect to the normalised measure, so that $(J_nb)^\wedge_\mu=\hat J_n(\mu)\hat b_\mu$.
\begin{enumerate}[label=(\roman*)]
    \item $J_n\ge0$ is a trigonometric polynomial of degree $2n-2$ in each variable with $\int_{\TT^d}J_n\dint y=1$, and $0\le\hat J_n(\mu)\le1$ for all $\mu\in\ZZ^d$.
    \item $\frac1{2\pi}\int_{-\pi}^{\pi}s^2j_n(s)\dint s\le\pi^3/n^2$, and consequently $0\le1-\hat J_n(\mu)\le\pi^3\abs\mu^2/(2n^2)$.
    \item For $b\in A^{1+\theta}(\TT^d)$, $\theta\in(0,1]$: $\norm{b-J_nb}_\infty\le\norm{b-J_nb}_{A^0}\le(\pi^3/2)^{(1+\theta)/2}n^{-(1+\theta)}\norm b_{A^{1+\theta}}$.
    \item For $b\in W^{1,\infty}(\TT^d)$: $\norm{b-J_nb}_\infty\le\pi^{3/2}\sqrt d\,\norm{\nabla b}_\infty/n$.
\end{enumerate}
\end{lemma}

\begin{proof}
\emph{(i)} By the Fej\'er identity $(\sin(ns/2)/\sin(s/2))^2=\sum_{\abs l<n}(n-\abs l)\me^{ils}$, the fourth power is the square of a trigonometric polynomial of degree $n-1$ with nonnegative coefficients; so it is nonnegative, of degree $2n-2$, with nonnegative Fourier coefficients and mean $\sum_{\abs l<n}(n-\abs l)^2=(2n^3+n)/3=\lambda_n$ (Parseval).
Hence $j_n\ge0$, $\hat\jmath_n\ge0$, $\hat\jmath_n(0)=1$ and $\hat\jmath_n(k)\le\frac1{2\pi}\int j_n=1$; the product structure transfers this to $J_n$.

\emph{(ii)} For $\abs s\le\pi$ we have $\sin(s/2)\ge\abs s/\pi$ and $\sin^4(ns/2)\le\min\{1,(ns/2)^4\}$, so $j_n(s)\le\lambda_n^{-1}\pi^4s^{-4}\min\{1,(ns/2)^4\}$ and
\[
    \frac1{2\pi}\int_{-\pi}^{\pi}s^2j_n(s)\dint s\le\frac{\pi^3}{2\lambda_n}\int_{\RR}\min\Bigl\{\frac1{s^2},\frac{n^4s^2}{16}\Bigr\}\dint s=\frac{\pi^3}{2\lambda_n}\cdot\frac{4n}3\le\frac{\pi^3}{n^2},
\]
using $\lambda_n\ge2n^3/3$.
Then $1-\hat\jmath_n(k)=\frac1{2\pi}\int(1-\cos ks)\,j_n(s)\dint s\le\frac{k^2}2\cdot\frac1{2\pi}\int s^2j_n\le\pi^3k^2/(2n^2)$, and since all factors lie in $[0,1]$, $1-\prod_i\hat\jmath_n(\mu_i)\le\sum_i(1-\hat\jmath_n(\mu_i))\le\pi^3\abs\mu^2/(2n^2)$.

\emph{(iii)} By (ii) and $\min\{1,x\}\le x^{(1+\theta)/2}$,
\[
    \norm{b-J_nb}_{A^0}=\sum_\mu\bigl(1-\hat J_n(\mu)\bigr)\abs{\hat b_\mu}\le\sum_\mu\min\Bigl\{1,\frac{\pi^3\abs\mu^2}{2n^2}\Bigr\}\abs{\hat b_\mu}\le\Bigl(\frac{\pi^3}{2n^2}\Bigr)^{(1+\theta)/2}\sum_\mu\abs\mu^{1+\theta}\abs{\hat b_\mu},
\]
and the last sum is at most $\norm b_{A^{1+\theta}}$.

\emph{(iv)} $\abs{b(x)-J_nb(x)}\le\int\abs{b(x-y)-b(x)}J_n(y)\dint y\le\norm{\nabla b}_\infty\int\abs yJ_n(y)\dint y$ with $y\in[-\pi,\pi)^d$, and by the Cauchy--Schwarz inequality and (ii), $\int\abs yJ_n\le(\int\abs y^2J_n)^{1/2}=\bigl(d\cdot\frac1{2\pi}\int s^2j_n\bigr)^{1/2}\le\pi^{3/2}\sqrt d/n$.
\end{proof}

\end{document}